\documentclass[10pt]{amsart}
\usepackage[cp1251]{inputenc}
\usepackage[english,russian]{babel}
\usepackage{amsmath}
\usepackage{amssymb}
\usepackage{amsfonts}

\newtheorem{lemma}{Lemma}
\newtheorem{theorem}{Theorem}

\newtheorem{corollary}{Corollary}
\newtheorem{proposition}{Proposition}

\usepackage{amscd,euscript,amsxtra,amsthm}

\theoremstyle{definition}
\newtheorem{definition}{Definition}
\theoremstyle{convention}

\theoremstyle{example}

\theoremstyle{remark}
\newtheorem{remark}{Remark}

\usepackage[all]{xy}

\numberwithin{equation}{section}

\def\A{{{\mathbb A}}}
\def\I{{{\mathbb I}}}

\def\G{{{\mathbb G }}}
\def\E{{{\mathbb E }}}

\def\Z{{{\mathbb Z }}}
\def\P{{{\mathbb P }}}

\def\L{{{\mathbb L }}}

\def\U{{{\mathbb U}}}
\def\V{{{\mathbb V }}}

\def\CC{{{\mathcal C}}}

\def\EE{{{\mathcal E}}}
\def\FF{{{\mathcal F}}}

\def\II{{{\mathcal I}}}

\def\LL{{{\mathcal L}}}
\def\MM{{{\mathcal M}}}
\def\NN{{{\mathcal N}}}
\def\OO{{{\mathcal O}}}
\def\QQ{{{\mathcal Q}}}

\def\TT{{{\mathcal T}}}

\def\WW{{{\mathcal W}}}

\def\EExt{{{\mathcal E}xt}}

\def\FFitt{{{\mathcal F}itt}}

\def\Spec{{{\rm Spec \,}}}
\def\Sing{{{\rm Sing \,}}}

\def\Proj{{{\rm Proj \,}}}
\def\Hilb{{{\rm Hilb \,}}}
\def\Univ{{{\rm Univ \,}}}
\def\Grass{{{\rm Grass }}}
\def\Quot{{{\rm Quot \,}}}
\def\Supp{{{\rm Supp \,}}}
\def\Pic{{{\rm Pic \,}}}

\def\dim{{{\rm dim \,}}}
\def\codim{{{\rm codim \,}}}
\def\hd{{{\rm hd \,}}}
\def\coker{{{\rm coker \,}}}
\def\ker{{{\rm ker \,}}}
\def\rank{{{\rm rank \,}}}
\def\im{{{\rm im \,}}}
\def\id{{{\rm id }}}

\def\tors{{{\rm tors \,}}}

\def\red{{{\rm red}}}

\begin{document}
\renewcommand{\refname}{References}
\renewcommand{\proofname}{Proof.}
\thispagestyle{empty}

\title[Admissible Pairs vs Gieseker--Maruyama]{Admissible Pairs vs Gieseker--Maruyama}
\author{{N.V.Timofeeva}}
\address{Nadezda Vladimirovna Timofeeva
\newline\hphantom{iii} Yaroslavl State University,
\newline\hphantom{iii} ul. Sovetskaya, 14,
\newline\hphantom{iii} 150000, Yaroslavl, Russia}%
\email{ntimofeeva@list.ru}%
\maketitle{\footnote{This work is done as a part of
the initiative project <<Mathematical methods of optimization
in continuous and discrete systems>> (NIR VIP-008) of Yaroslavl
State University}\vspace{1cm} \small
\begin{quote}
\noindent{\sc Abstract. } A morphism of the moduli functor  of
admissible semistable pairs  to the  Gieseker -- Maruyama moduli
functor (of semistable coherent torsion-free sheaves)  with the
same Hilbert polynomial on the surface, is constructed. It is
shown that these functors are isomorphic, and the moduli scheme
for semistable admissible pairs $((\widetilde S, \widetilde L),
\widetilde E)$ is isomorphic to  the Gieseker -- Maruyama moduli
scheme. The considerations involve all components of moduli
functors and corresponding moduli scheme as they exist.
\medskip

 \end{quote}

\bigskip

\begin{flushright}
{\it To the blessed memory of my Mom}
\end{flushright}

\section*{Introduction}
In this article we complete the investigation of the compactification of
moduli of stable vector bundles on a surface by locally free
sheaves by examining the whole of moduli of admissible semistable
pairs. Various aspects of the construction of main components
(constituting the compactification of moduli of vector bundles)
 and basic properties were given in preceding papers
of the author \cite{Tim0} -- \cite{Tim9}. In the cited articles
the key restriction was that all families under
consideration include so-called $S$-pairs. In this article we
eliminate this restriction.

Let $S$ be a smooth irreducible projective algebraic surface over
a field $k=\overline k$ of zero characteristic, $\OO_S$ its
structure sheaf, $E$ coherent torsion-free $\OO_S$-module,
$E^{\vee}:={{\mathcal H}om}_{{\mathcal O}_S}(E, {\mathcal O}_S)$
its dual ${\mathcal O}_S$-module. $E^{\vee}$ is reflexive and
hence locally free. A locally free sheaf and its corresponding
vector bundle are canonically identified and both terms are used
as synonyms. Let  $L$ be very ample invertible sheaf on $S$; it is
fixed and is used as a polarization. The symbol $\chi(\cdot)$
denotes Euler -- Poincar\'{e} characteristic, $c_i(\cdot)$
 $i$-th Chern class.

\begin{definition} \label{admsch} \cite{Tim3, Tim4} Polarized algebraic scheme
$(\widetilde S, \widetilde L)$ is called {\it admissible} if it
satisfies one of the following conditions

i) $(\widetilde S, \widetilde L) \cong (S,L)$,

ii) $\widetilde S \cong {\rm Proj \,} \bigoplus_{s\ge
0}(I[t]+(t))^s/(t^{s+1})$ where $I={{\mathcal F}itt}^0 {{\mathcal
E}xt}^2(\varkappa, {\mathcal O}_S)$ for Artinian quotient sheaf
 $q_0: \bigoplus^r {\mathcal
O}_S\twoheadrightarrow \varkappa$ of length $l(\varkappa)\le c_2$,
and $\widetilde L = L \otimes (\sigma ^{-1} I \cdot {\mathcal
O}_{\widetilde S})$ is very ample invertible sheaf on the scheme
$\widetilde S$; this polarization  $\widetilde L$ is called  {\it
distinguished polarization}.
\end{definition}

Now discuss and/or recall several notions and objects involved in
this definition.

Recall the definition of a sheaf of 0-th Fitting ideals known from
commutative algebra. Let $X$ be a scheme, $F$ $\OO_X$-module of
finite presentation $F_1 \stackrel{\varphi}{\longrightarrow} F_0
\to F$. Without loss of generality we assume that  $\rank F_1 \ge
\rank F_0$.
\begin{definition} {\it The sheaf of 0-th Fitting ideals } of
$\OO_X$-module $F$ is defined as  $\FFitt^0 F =\im
(\bigwedge^{\rank F_0} F_1 \otimes \bigwedge^{\rank F_0}
F_0^{\vee} \stackrel{\varphi'}{\longrightarrow}\OO_X)$, where
$\varphi'$ is a morphism of  $\OO_X$-modules induced by $\varphi$.
\end{definition}

\begin{remark} In further considerations we replace $L$ by its
big enough tensor power, if necessary for $\widetilde L$ to be
very ample. This power can be chosen uniform and fixed, as shown
in \cite{Tim4}. All Hilbert polynomials are compute according to
new $L$ and $\widetilde L$ respectively.
\end{remark}

As shown in \cite{Tim3}, if $\widetilde S$ satisfies the condition
(ii) in the definition \ref{admsch}, it is decomposed into the
union of several components $\widetilde S=\bigcup_{i\ge
0}\widetilde S_i$. It has a morphism $\sigma: \widetilde S \to S$
which is induced by the structure of $\OO_S$-algebra on the graded
object $\bigoplus_{s\ge 0}(I[t]+(t))^s/(t^{s+1})$. The scheme
$\widetilde S$ can be produced as follows. Take a product $\Spec
k[t] \times S$ and its blowing up $Bl_{\II}\Spec k[t] \times S$ in
the sheaf of ideals $\II=(t)+I[t]$ corresponding to the subscheme
with ideal $I$ in the zero-fibre $0\times S$. If
$\sigma\!\!\!\sigma: Bl_{\II}\Spec k[t] \times S \to {\Spec k[t]
\times S}$ is the blowup morphism then $\widetilde S$ is a
zero-fibre of the composite $$pr_1\circ \sigma\!\!\!\sigma:
Bl_{\II}\Spec k[t] \times S \to {\Spec k[t] \times S} \to \Spec
k[t].$$

\begin{definition} \cite{Tim4} $S$-{\it stable}
(respectively, {\it semistable}) {\it pair} $((\widetilde
S,\widetilde L), \widetilde E)$ is the following data:
\begin{itemize}
\item{$\widetilde S=\bigcup_{i\ge 0} \widetilde S_i$ --
admissible scheme, $\sigma: \widetilde S \to S$ morphism which is
called {\it canonical}, $\sigma_i: \widetilde S_i \to S$ its
restrictions on components $\widetilde S_i$, $i\ge 0;$}
\item{$\widetilde E$ vector bundle on the scheme
$\widetilde S$;}
\item{$\widetilde L \in \Pic \widetilde S$ distinguished polarization;}
\end{itemize}
such that
\begin{itemize}
\item{$\chi (\widetilde E \otimes \widetilde
L^n)=rp(n),$ the polynomial $p(n)$ and the rank $r$ of the sheaf
$\widetilde E$ are fixed;}
\item{the sheaf $\widetilde E$ on the scheme $\widetilde S$ is {\it
stable } (respectively, {\it semistable}) {\it due to Gieseker,}
i.e. for any proper subsheaf $\widetilde F \subset \widetilde E$
for $n\gg 0$
\begin{eqnarray*}
\frac{h^0(\widetilde F\otimes \widetilde L^n)}{{\rm rank\,} F}&<&
\frac{h^0(\widetilde E\otimes \widetilde L^n)}{{\rm rank\,} E},
\\ (\mbox{\rm respectively,} \;\;
\frac{h^0(\widetilde F\otimes \widetilde L^n)}{{\rm rank\,}
F}&\leq& \frac{h^0(\widetilde E\otimes \widetilde L^n)}{{\rm
rank\,} E}\;);
\end{eqnarray*}}
\item{on each of additional components $\widetilde S_i, i>0,$
the sheaf $\widetilde E_i:=\widetilde E|_{\widetilde S_i}$ is {\it
quasi-ideal,} i.e. admits a description of the form
\begin{equation}\label{quasiideal}\widetilde E_i=\sigma_i^{\ast}
\ker q_0/tor\!s_i.\end{equation} for some $q_0\in \bigsqcup_{l\le
c_2} {\rm Quot\,}^l \bigoplus^r {\mathcal O}_S$. }\end{itemize}
\end{definition}
The definition of the subsheaf $tor\!s_i$ will be given below.

Pairs $(( \widetilde S, \widetilde L), \widetilde E)$ such that
 $(\widetilde S, \widetilde
L)\cong (S,L)$ will be called {\it $S$-pairs}.

In the series of articles of the author \cite{Tim0}
--- \cite{Tim4} a projective algebraic scheme
$\widetilde M$ is built up as reduced moduli scheme of
$S$-semistable admissible pairs and in \cite{Tim6} it is
constructed as possibly nonreduced moduli space.

The scheme  $\widetilde M$ contains an open subscheme  $\widetilde
M_0$ which is isomorphic to the subscheme $M_0$ of
Gieseker-semistable vector bundles in the Gieseker -- Maruyama
moduli scheme $\overline M$ of torsion-free semistable sheaves
whose Hilbert polynomial is equal to  $\chi(E \otimes L^n)=rp(n)$.
The following definition of Gieseker-semistability is used.

\begin{definition}\cite{Gies} The coherent  ${\mathcal O}_S$-sheaf $E$ is {\it
stable} (respectively, {\it semistable}) if for any proper
subsheaf  $F\subset E$ of rank $r'={\rm rank\,} F$ for $n\gg 0$
$$
\frac{\chi(E \otimes L^n)}{r}>\frac{\chi(F\otimes L^n)}{r'},\;
{\mbox{\LARGE (}}{\mbox{\rm respectively,}} \; \frac{\chi(E
\otimes L^n)}{r}\ge \frac{\chi(F\otimes L^n)}{r'}{\mbox{\LARGE
)}}.
$$
\end{definition}

Let  $E$ be a semistable locally free sheaf. Then, obviously, the
sheaf  $I={{\mathcal F}itt}^0 {{\mathcal E}xt}^1(E, {\mathcal
O}_S)$ is trivial and $\widetilde S \cong S$. In this case
$((\widetilde S, \widetilde L), \widetilde E) \cong ((S, L),E)$
and we have a bijective correspondence $\widetilde M_0 \cong M_0$.

Let $E$ be a semistable nonlocally free coherent sheaf; then the
scheme  $\widetilde S$ contains reduced irreducible component
$\widetilde S_0$ such that the morphism
$\sigma_0:=\sigma|_{\widetilde S_0}: \widetilde S_0 \to S$ is a
morphism of blowing up of the scheme $S$ in the sheaf of ideals
$I= {{\mathcal F}itt}^0 {{\mathcal E}xt}^1(E, {\mathcal O}_S).$
Formation of a sheaf  $I$ is an approach to the characterization
of singularities of the sheaf $E$ i.e. its difference from a
locally free sheaf. Indeed, the quotient sheaf $\varkappa:=
E^{\vee \vee}/ E$ is Artinian of length not greater then $c_2(E)$,
and ${{\mathcal E}xt}^1(E, {\mathcal O}_S) \cong {{\mathcal
E}xt}^2(\varkappa, {\mathcal O}_S).$ Then ${{\mathcal F}itt}^0
{{\mathcal E}xt}^2(\varkappa, {\mathcal O}_S)$ is a sheaf of
ideals of  (in general case nonreduced) subscheme $Z$ of bounded
length \cite{Tim6} supported at finite set of points on the
surface $S$. As it is shown in \cite{Tim3}, others irreducible
components $\widetilde S_i, i>0$ of the scheme  $\widetilde S$ in
general case carry nonreduced scheme structure.

Each semistable coherent torsion-free sheaf $E$ corresponds to a
pair  $((\widetilde S, \widetilde L), \widetilde E)$ where
$(\widetilde S, \widetilde L)$  defined as described.

Now we describe the construction of the subsheaf $tor\!s$ in
(\ref{quasiideal}). Let $U$ be Zariski-open subset in one of
components  $\widetilde S_i, i\ge 0$, and
$\sigma^{\ast}E|_{\widetilde S_i}(U)$ correspond\-ing group of
sections. This group is ${\mathcal O}_{\widetilde S_i}(U)$-module.
Sections  $s\in \sigma^{\ast}E|_{\widetilde S_i}(U)$ annihilated
by prime ideals of positive codimensions in  ${\mathcal
O}_{\widetilde S_i}(U)$, form a submodule in
$\sigma^{\ast}E|_{\widetilde S_i}(U)$. This submodule is denoted
as  $tor\!s_i(U)$. The correspondence
 $U \mapsto tor\!s_i(U)$ defines a subsheaf $tor\!s_i
\subset \sigma^{\ast}E|_{\widetilde S_i}.$ Note that associated
primes of positive codimensions which annihilate sections $s\in
\sigma^{\ast}E|_{\widetilde S_i}(U)$, correspond to subschemes
supported in the preimage $\sigma^{-1}({\rm Supp\,}
\varkappa)=\bigcup_{i>0}\widetilde S_i.$ Since by the construction
the scheme  $\widetilde S=\bigcup_{i\ge 0}\widetilde S_i$ is
connected \cite{Tim3}, subsheaves  $tor\!s_i, i\ge 0,$ allow to
construct a subsheaf $tor\!s \subset \sigma^{\ast}E$. The former
subsheaf is defined as follows. A section  $s\in
\sigma^{\ast}E|_{\widetilde S_i}(U)$ satisfies the condition $s\in
tor\!s|_{\widetilde S_i}(U)$ if and only if
\begin{itemize}
\item{there exist a section  $y\in {\mathcal O}_{\widetilde S_i}(U)$ such that
 $ys=0$,}
\item{at least one of the following two conditions is satisfied:
either  $y\in {\mathfrak p}$, where $\mathfrak p$ is prime ideal
of positive codimension; or there exist Zariski-open subset
$V\subset \widetilde S$ and a section  $s' \in \sigma^{\ast}E (V)$
such that  $V\supset U$, $s'|_U=s$, and $s'|_{V\cap \widetilde
S_0} \in$\linebreak $tor\!s (\sigma^{\ast}E|_{\widetilde
S_0})(V\cap \widetilde S_0)$. In the former expression the torsion
subsheaf $tor\!s(\sigma^{\ast}E|_{\widetilde S_0})$ is understood
in usual sense.}
\end{itemize}

The role of the subsheaf $tor\!s \subset \sigma^{\ast}E$ in our
construction is analogous to the role of torsion subsheaf in the
case of reduced and irreducible base scheme. Since no confusion
occur, the symbol  $tor\!s$ is understood everywhere in described
sense. The subsheaf $tor\!s$ is called a {\it torsion subsheaf}.

In \cite{Tim4} it is proven that sheaves $\sigma^{\ast} E/tor\!s$
are locally free. The sheaf $\widetilde E$ include in the pair
$((\widetilde S, \widetilde L), \widetilde E)$ is defined by the
formula $\widetilde E = \sigma^{\ast}E/tor\!s$. In this
circumstance there is an isomorphism $H^0(\widetilde S, \widetilde
E \otimes \widetilde L) \cong H^0(S,E\otimes L).$

In the same article it was proven that the restriction of the
sheaf  $\widetilde E$ to each of components $\widetilde S_i$,
$i>0,$ is given by the quasi-ideality relation (\ref{quasiideal})
where $q_0: {\mathcal O}_S^{\oplus r}\twoheadrightarrow \varkappa$
is an epimorphism defined by the exact triple  $0\to E \to E^{\vee
\vee} \to \varkappa \to 0$ in view of local freeness of the sheaf
$E^{\vee \vee}$.

Resolution of singularities of a semistable sheaf $E$ can be
globalized in a flat family by means of the construction developed
in various versions in  \cite{Tim1, Tim2, Tim4, Tim8}. Let $T$ be
a  scheme, ${\mathbb E}$ a sheaf of  ${\mathcal O}_{T\times
S}$-modules, ${\mathbb L}$ invertible ${\mathcal O}_{T\times
S}$-sheaf very ample relative to $T$ and such that  ${\mathbb
L}|_{t\times S}=L$, and $\chi({\mathbb E} \otimes {\mathbb
L}^n|_{t\times S})=rp(n)$ for all closed points $t\in T$. We also
assume that $\E$ and $\L$ are flat relative to $T$ and $T$
contains nonempty open subset $T_0$ such that ${\mathbb
E}|_{T_0\times S}$ is locally free ${\mathcal O}_{T_0 \times
S}$-module. Then following objects are defined:
\begin{itemize}
\item{$\pi: \widetilde \Sigma \to \widetilde T$ flat family of admissible
schemes with invertible ${\mathcal O}_{\widetilde \Sigma}$-module
$\widetilde {\mathbb L}$ such that  $\widetilde {\mathbb
L}|_{t\times S}$ distinguished polarization of the scheme
$\pi^{-1}(t)$,}
\item{$\widetilde {\mathbb E}$ locally free
${\mathcal O}_{\widetilde \Sigma}$-module and $((\pi^{-1}(t),
\widetilde {\mathbb L}|_{\pi^{-1}(t)}),  \widetilde {\mathbb
E}|_{\pi^{-1}(t)})$ is $S$-semistable admiss\-ible pair.}
\end{itemize}
In this situation there is a blowup morphism $\Phi: \widetilde
\Sigma \to
 \widetilde T \times S$.

The mechanism described was called a {\it standard resolution}.

\smallskip

In the present article we prove  following results.
\begin{theorem}\label{thfunc} (i) There is a natural transformation
$\underline \kappa: {\mathfrak f}^{GM} \to {\mathfrak f}$ of
Gieseker -- Maruyama moduli functor  to moduli functor of
admissible semistable pairs with same rank and Hilbert polynomial.

(ii) There is a natural transformation $\underline \tau:
{\mathfrak f} \to {\mathfrak f}^{GM}$ of the moduli functor of
admissible semistable pairs  to Gieseker --
Maruyama moduli functor for sheaves with same rank and Hilbert
polynomial.

(iii) Natural transformations $\underline \kappa$ and $\underline
\tau$ are mutually inverse.
 Hence both morphisms of nonreduced moduli functors $\underline \kappa:
\mathfrak f^{GM}\to \mathfrak f$ and $\underline \tau: {\mathfrak
f} \to {\mathfrak f}^{GM}$ are isomorphisms.
\end{theorem}
\begin{corollary}\label{thsch}
The  nonreduced moduli scheme $\widetilde M$ for ${\mathfrak f}$
is isomorphic to the  nonreduced Gieseker -- Maruyama scheme
$\overline M$ for sheaves with same rank and Hilbert polynomial.
\end{corollary}
In section 1 we remind definitions of the functor  $\mathfrak
f^{GM}$ of moduli of coherent torsion-free sheaves ("Gieseker --
Maruyama functor") (\ref{funcGM}, \ref{famGM}) and of the functor
$\mathfrak f$ of moduli of admissible semistable pairs
(\ref{funcmy}, \ref{class}). The rank $r$ and polynomial $p(n)$
are fixed and equal for both moduli functors.

Then in section 2 we give the
transformation of the family of coherent torsion-free sheaves $(T,
\L, \E)$ with base scheme $T$ into a family of admissible semistable pairs $((T, \pi:
\widetilde \Sigma \to T, \widetilde \L),\widetilde \E)$. This
transformation generalizes the procedure of standard resolution
for the case when the initial family is not obliged to contain
locally free sheaves. It leads to the functorial morphism
$\underline \kappa: {\mathfrak f}^{GM} \to {\mathfrak f}$ and proves
part {\it i)} of the Theorem \ref{thfunc}.

After that, in section 3 we give the description of the
transformation of a family of semistable admissible pairs
 $((\pi: \widetilde \Sigma \to T, \widetilde \L),
\widetilde \E)$ with (possibly, nonreduced) base scheme $T$ to a
family $\E$ of coherent torsion-free semistable sheaves with the
same base $T$. The transformation provides a morphism of the
functor $\underline \tau$ of admissible semistable pairs ${\mathfrak f}$ to Gieseker
-- Maruyama functor $\mathfrak f^{GM}$ and proves part {\it ii)} of Theorem \ref{thfunc}.

In section 4 we show that  morphisms of functors $\underline \kappa: {\mathfrak
f}^{GM} \to {\mathfrak f}$ and $\underline \tau:{\mathfrak f} \to  {\mathfrak
f}^{GM}$ we constructed are mutually inverse. In this way
the functors of interest are isomorphic and this completes the proof
of Theorem \ref{thfunc}.

\smallskip

\section{Moduli functors}

Following \cite[ch. 2, sect. 2.2]{HL}, we recall some definitions.
Let ${\mathcal C}$ be a category, ${\mathcal C}^o$ its dual,
 ${\mathcal C}'={{\mathcal F}unct}({\mathcal C}^o, Sets)$ category of
 functors to the category of sets. By Yoneda's lemma, the functor  ${\mathcal C} \to
{\mathcal C}': F\mapsto (\underline F: X\mapsto {\rm Hom
\,}_{{\mathcal C}}(X, F))$ includes ${\mathcal C}$ into ${\mathcal
C}'$ as full subcategory.

\begin{definition}\label{corep}\cite[ch. 2, definition 2.2.1]{HL}
The functor  ${\mathfrak f} \in {\OO}b\, \CC'$ is {\it
corepresented by the object} $M \in {\OO}b \,\CC$, if there exist
a $\CC'$-morphism $\psi : {\mathfrak f} \to \underline M$ such
that any morphism $\psi': {\mathfrak f} \to \underline F'$ factors
through the unique morphism  $\omega: \underline M \to \underline
F'$.
\end{definition}

\begin{definition} The scheme  $\widetilde M$ is a {\it coarse moduli space}
for the functor $\mathfrak f$ if  $\mathfrak f$ is corepresented
by the scheme
 $\widetilde M.$
 \end{definition}

Let $T,S$ be  schemes over a field $k$, $\pi: \widetilde \Sigma
\to T$ a morphism of $k$-schemes. We introduce the following

\begin{definition}\label{bitriv} The family of schemes
$\pi: \widetilde \Sigma \to T$ is {\it birationally $S$-trivial}
if there exist isomorphic open subschemes $\widetilde \Sigma_0
\subset \widetilde \Sigma$ and $\Sigma_0 \subset T\times S$ and
there is a scheme equality $\pi(\widetilde \Sigma_0)=T$.
\end{definition}
The former equality means that all fibres of the morphism $\pi$
have nonempty intersections with the open subscheme $\widetilde
\Sigma_0$.

In particular, if $T=\Spec k$ then $\pi$ is a constant morphism
and $\widetilde \Sigma_0 \cong \Sigma_0$ is open subscheme in $S$.

Since in the present paper we consider only $S$-birationally
trivial families, they will be referred to as {\it birationally
trivial} families.

We consider sets of families of semistable pairs
\begin{equation}\label{class}{\mathfrak F}_T= \left\{
\begin{array}{l}\pi: \widetilde \Sigma \to T \mbox{\rm \;\;birationally $S$-trivial},\\
\widetilde \L\in \Pic \widetilde \Sigma \mbox{\rm \;\;flat over
}T,\\
\mbox{\rm for }m\gg 0 \; \widetilde \L^m \; \mbox{\rm very ample
relatively }T,\\ \forall t\in T \;\widetilde L_t=\widetilde
\L|_{\pi^{-1}(t)}
\mbox{\rm \; ample;}\\
(\pi^{-1}(t),\widetilde L_t) \mbox{\rm \;admissible scheme with
distinguished polarization}; \\
\chi (\widetilde L_t^n) \mbox{\rm \; does not depend on }t,\\
 \widetilde \E \;\; \mbox{\rm locally free } \OO_{\Sigma}-\mbox{\rm
sheaf flat over } T;\\
 \chi(\widetilde \E\otimes\widetilde \L^{n})|_{\pi^{-1}(t)})=
 rp(n);\\
 ((\pi^{-1}(t), \widetilde L_t), \widetilde \E|_{\pi^{-1}(t)}) -
 \mbox{\rm semistable pair}
 \end{array} \right\} \end{equation}  and a functor
\begin{equation}\label{funcmy}\mathfrak f: (Schemes_k)^o
\to (Sets)\end{equation} from the category of $k$-schemes to the
category of sets. It attaches to any scheme $T$ the set of
equivalence classes of families of the form $({\mathfrak
F}_T/\sim).$

The equivalence relation  $\sim$ is defined as follows. Families
$((\pi: \widetilde \Sigma \to T, \widetilde \L), \widetilde \E)$
and \linebreak $((\pi': \widetilde \Sigma \to T, \widetilde \L'),
\widetilde
 \E')$ from the class  $\mathfrak F_T$ are said to be equivalent
 (notation:\linebreak
 $((\pi: \widetilde \Sigma \to T, \widetilde \L),
 \widetilde \E) \sim ((\pi': \widetilde \Sigma \to T, \widetilde \L'), \widetilde
 \E')$) if\\
 1) there exist an isomorphism $ \iota: \widetilde \Sigma \stackrel{\sim}{\longrightarrow}
 \widetilde \Sigma'$ such that the diagram \begin{equation*}
 \xymatrix{\widetilde \Sigma  \ar[rd]_{\pi}\ar[rr]_{\sim}^{\iota}&&\widetilde \Sigma ' \ar[ld]^{\pi'}\\
&T }
 \end{equation*} commutes.\\
 2) There exist line bundles   $L', L''$ on the scheme  $T$ such
 that $\iota^{\ast}\widetilde \E' = \widetilde \E \otimes \pi^{\ast} L',$
 $\iota^{\ast}\widetilde \L' = \widetilde \L \otimes \pi^{\ast} L''.$


Now discuss what is the "size"\, of the maximal under  inclusion
of those open sub\-schemes $\widetilde \Sigma_0$ in a family of
admissible schemes $\widetilde \Sigma$, which are isomorphic to
appropriate open subschemes in $T\times S$ in the definition
\ref{bitriv}. The set $F=\widetilde \Sigma \setminus \widetilde
\Sigma_0$ is closed. If $T_0$ is open subscheme in $T$ whose
points carry fibres isomorphic to $S$, then $\widetilde \Sigma_0
\supsetneqq \pi^{-1}T_0$ (inequality is true because $\pi
(\widetilde \Sigma_0)=T$ in the definition \ref{bitriv}). The
subscheme $\Sigma_0$ which is open in $T\times S$ and isomorphic
to $\widetilde \Sigma_0$, is such that $\Sigma_0 \supsetneqq
T_0\times S$. If $\pi:\widetilde \Sigma \to T$ is family of
admissible schemes then $\widetilde \Sigma_0 \cong \widetilde
\Sigma \setminus F$, and $F$ is (set-theoretically) the union of
additional components of fibres which are non-isomorphic to $S$.
Particularly, this means that $\codim_{T\times S}
(T\times S)\setminus \Sigma_0 \ge 2.$


\smallskip

The Gieseker -- Maruyama functor
\begin{equation}\label{funcGM}{\mathfrak f}^{GM}:
(Schemes_k)^o \to Sets,\end{equation} attaches to any scheme $T$
the set of equivalence classes of families of the following form
${\mathfrak F}_T^{GM}/\sim$, where
\begin{equation}\label{famGM} \mathfrak
F_T^{\,GM}= \left\{
\begin{array}{l} \E \;\;\mbox{\rm sheaf of } \OO_{T\times S}-
\mbox{\rm modules flat over } T;\\
\L \;\;\mbox{\rm invertible sheaf of } \OO_{T\times S}-\mbox{\rm
modules,}\\ \mbox{\rm  ample relatively to } T\\
\mbox{\rm and such that } L_t:=\L|_{t\times S}\cong L\; \mbox{\rm for any point } t\in T;\\
E_t:=\E|_{t\times S} \;\mbox{\rm torsion-free and Gieseker-semistable;}\\
\chi(E_t \otimes L_t^n)=rp(n).\end{array}\right\}
\end{equation}

Families   $\E, \L$ and $\E',\L'$ from the class $\mathfrak
F^{GM}_T$ are said to be equivalent (notation: $ (\E, \L)\sim
(\E',\L')$), if there exist linebundles  $L', L''$ on the scheme
$T$ such that $\E' = \E \otimes p^{\ast} L',$ $ \L' =  \L \otimes
p^{\ast} L''$ where $p: T\times S \to T$ is projection onto the
first factor.

\begin{remark} Since $\Pic (T \times S) =\Pic T \times \Pic
S$, our definition of the moduli functor ${\mathfrak f}^{GM}$ is
equivalent to the standard definition which can be found, for
example, in \cite{HL}: the difference in choice of polarizations
$\L$ and $\L'$ having isomorphic restrictions on fibres over the
base $T$, is avoided by the equivalence which is induced by
tensoring by inverse image of an invertible sheaf $L''$ from the
base $T$.
\end{remark}

\section{GM-to-Pairs transformation (standard resolution)}

The morphism of functors  $\underline \kappa: \mathfrak f^{GM} \to
\mathfrak f$ is defined by commutative diagrams
\begin{equation}\label{morfun}\xymatrix{T \ar@{|->}[rd] \ar@{|->}[r]&
\mathfrak F^{GM}_T/\sim \ar[d]\\
& \mathfrak F_T/\sim}
\end{equation}
where $T\in {\mathcal O}b (Schemes)_k$, $\underline
\kappa(T):(\mathfrak F^{GM}_T/\sim) \to ( \mathfrak F_T/\sim)$ is
a morphism in the category of sets (mapping).

The aim of this section is to build up a transformation of the
family $(T, \L, \E)$ of semistable coherent torsion-free sheaves
to the family $((T, \pi: \widetilde \Sigma \to T, \widetilde \L),
\widetilde \E)$ of admissible semistable pairs. Since the initial family of sheaves $\E$ is not
obliged to contain at least one locally free sheaf then
codimension of singular locus $\Sing \E$ in $\Sigma=T\times S$ can
equal 2. Hence if the
blowing up $\sigma\!\!\!\sigma:\widetilde \Sigma \to \Sigma$ of
the sheaf of ideals $\FFitt^0 \EExt^1(\E, \OO_{\Sigma})$ is
considered then the fibres of the composite $p\circ
\sigma\!\!\!\sigma$ are not obliged to be equidimensional. Such a blowing up
cannot produce family of admissible schemes.

To overcome this difficulty we perform the following artifical but
obvious trick. Consider the product $\Sigma'=\Sigma \times \A^1$
and fix a closed immersion $i_0: \Sigma \hookrightarrow \Sigma'$
which identifies $\Sigma$ with zero fibre $\Sigma \times 0.$ Now
let $Z\subset \Sigma$  be a subscheme defined by the sheaf of
ideals $\I=\FFitt^0 \EExt^1(\E, \OO_{\Sigma})$. Then consider the
sheaf of ideals $\I':=\ker (\OO_{\Sigma'} \twoheadrightarrow
i_{0\ast} \OO_Z)$ and the blowup morphism $\sigma\!\!\!\sigma':
\widehat \Sigma' \to \Sigma'$ defined by the sheaf $\I'$. Denote
the projection onto the product of factors $\Sigma' \to
T\times \A^1=: T'$ by $p'$ and the composite $p' \circ
\sigma\!\!\!\sigma'$ by $\widehat \pi'$. We are interested in the
induced morphism $\pi: \widetilde \Sigma:=i_0(\Sigma)
\times_{\Sigma'} \widehat \Sigma' \to T$. Under
the identification $\Sigma \cong i_0(\Sigma)$ we denote by
$\sigma\!\!\!\sigma$ the induced morphism $\widehat \Sigma \to
\Sigma.$ Set $\L':= \L \boxtimes \OO_{\A^1}$. Obviously, there
exists $m\gg 0$ such that the invertible sheaf $\widehat \L':=
\sigma\!\!\!\sigma'^{\ast}{\L'}^m \otimes
({\sigma\!\!\!\sigma'}^{-1} \I') \cdot
\OO_{\widehat \Sigma'}$ is ample relatively to $\widehat \pi'$.
For brevity of notations we fix this $m$ and replace $L$ by its
$m$-th tensor product throughout further text. Denote
$\widetilde \L:=\widehat \L'|_{\widetilde \Sigma}.$
\begin{proposition} The morphism $\pi: \widetilde \Sigma \to
\Sigma$ is flat and fibrewise Hilbert polynomial compute with
respect to $\widetilde \L$, i.e. $\chi (\widetilde
\L^n|_{\pi^{-1}(t)})$, is uniform over $t\in T$.
\end{proposition}
\begin{proof}
First recall the following definition from \cite[$O_{III}$,
definition 9.1.1]{EGAIII}.
\begin{definition} The continuous mapping $f: X\to Y$ is called {\it
quasi-compact} if for any open quasi-compact subset $U\subset Y$
its preimage $f^{-1}(U)$ is quasi-compact. Subset  $Z$ is called
{\it retro-compact in } $X$ if the canonical injection
 $Z \hookrightarrow X$ is quasi-compact, and if for
 any open quasi-compact subset  $U \subset X$ the intersection $U
\cap Z$ is quasi-compact.
\end{definition}

Let  $f: X \to S$ be a scheme morphism of finite presentation,
$\MM$ be a quasi-coherent  $\OO_X$-module of finite type.

\begin{definition} \cite[part 1, definition 5.2.1]{RG} $\MM$
is $S$-{\it flat in dimension} $\ge n$ if there exist a
retro-compact open subset $V \subset X$ such that $\dim
(X\setminus V)/S < n$ and if  $\MM|_V$ is $S$-flat module of
finite presentation.
\end{definition}
If $\MM$ is $S$-flat  module of finite presentation and schemes
$X$ and $S$ are of finite type over the field, then any open
subset $V \subset X$ fits to be used in the definition. Setting
$V=X$ we have $X\setminus V=\emptyset$ and $\dim (X\setminus V)/S
=-1-\dim S.$ Consequently,
 $S$-flat  module of finite presentation is flat in dimension  $\ge -\dim S.$

Conversely, let  $\OO_X$-module $\MM$ be  $S$-flat in dimension
$\ge -\dim S$. Then there is an open retro-compact subset
$V\subset X$ such that  $\dim (X\setminus V)/S<-\dim S$ and such
that  $\MM|_V$ is  $S$-flat module. By the former inequality for
dimensions we have  $\dim (X\setminus V)<0$, what implies $X=V$,
and $\MM|_V=\MM$ is $S$-flat.

\begin{definition} \cite[part 1, definition 5.1.3]{RG} Let
 $f: S' \to S$ be a morphism of finite type, $U$ be an open
 subset in  $S$. The morphism $f$ is called
 $U$-{\it admissible blowup} if there exist a closed subscheme
 $Y \subset S$ of finite
 presentation  which is disjoint from $U$
 and such that  $f$ is isomorphic to the blowing up a scheme
$S$ in $Y$.
\end{definition}

\begin{theorem}\label{trg}\cite[theorem 5.2.2]{RG} Let $S$ be a
quasi-compact quasi-separated scheme, $U$ be open quasi-compact
subscheme in  $S$, $f: X\to S$ of finite presentation, $\MM$
$\OO_X$-module of finite type, $n$ an integer. Assume that
$\MM|_{f^{-1}(U)}$ is flat over $U$ in dimension $\ge n$. Then
there exist  $U$-admissible blowup $g: S' \to S$ such that
$g^{\ast} \MM$ is $S'$-flat in dimension $\ge n$.
\end{theorem}

Recall the following
\begin{definition}\cite[definition 6.1.3]{GroDieu} The scheme morphism
$f: X\to Y$ is {\it quasi-separated} if the diagonal morphism
$\Delta_f: X \to X\times_Y X$ is quasi-compact. The scheme  $X$ is
{\it quasi-separated} if it is quasi-separated over $\Spec \Z.$
\end{definition}
If the scheme  $X$ is Noetherian, then any morphism
 $f: X\to Y$ is quasi-compact. Since we work in the category
 of Noetherian schemes, all morphisms of our interest and all arising schemes
 are quasi-compact.

Set $f=\widehat \pi'$, $\MM=\OO_{\widehat \Sigma'},$
$U=T'\setminus T\times 0$. Hence by theorem \ref{trg}, there
exists a $T'\setminus T\times 0$-admissible blowing up $g:
\widetilde T' \to T'$ such that in the fibred square
\begin{equation}\label{ressq}\xymatrix{\widetilde \Sigma'
\ar[d]_{\widetilde\pi'} \ar[r]^{\widetilde g}&
\widehat \Sigma' \ar[d]^{\widehat \pi'}\\
\widetilde T' \ar[r]^g & T'}
\end{equation}
$\OO_{\widetilde \Sigma'}=\widetilde g^{\ast} \OO_{\widehat
\Sigma'}$ is flat $\OO_{\widetilde T'}$-module.

By reasonings and results of \cite[sect.3]{Tim8} and by
\cite[Prop.3]{Tim8} (the proof is applicable to the invertible
sheaves $\widetilde \L'$ and $\widehat \L'$ instead of $\widetilde
\L$ and $\widehat \L$ respectively), the morphism $\widehat \pi'$
is flat and computation of fibrewise Hilbert polynomials with
respect to $\widehat \L'$ leads to polynomials which are uniform
over the base $T'$. Then set $g=\id_{T'},$ $\widetilde
g=\id_{\widetilde \Sigma'},$ $\widetilde \Sigma'=\widehat
\Sigma'$.

Now $\pi: \widetilde \Sigma \to T$ is flat since it is obtained
from the flat morphism $\widehat \pi'$ by the base change.
\end{proof}
We denote $\sigma\!\!\!\sigma:=\sigma\!\!\!\sigma'|_{\widetilde
\Sigma}: \widetilde \Sigma \to \Sigma.$

To resolve singularities of the sheaf $\E$ we repeat all the
manipulations from \cite{Tim8} for the morphism
$\sigma\!\!\!\sigma$ as it was defined now.

Let $T$ be arbitrary  (possibly nonreduced) $k$-scheme of finite
type. We assume that its reduction  $T_{\red}$ is irreducible. If
$\E$ is a family of coherent torsion-free sheaves on the surface
$S$ having reduced base $T$ then homological dimension of  $\E$ as
$\OO_{T\times S}$-module is not greater then 1. The proof of this
fact for reduced equidimensional base can be found, for example,
in \cite[Proposition 1]{Tim0}.

Now we need the following simple lemma concerning homological
dimension of the family $\E$ with nonreduced base and proven in
\cite[Lemma 1]{Tim8}.
\begin{lemma}
Let coherent $\OO_{T\times S}$-module $\E$ of finite type is
$T$-flat  and its reduction $\E_{\red}:=\E \otimes_{\OO_{T}}
\OO_{T_{\red}}$ has homological dimension not greater then 1:
$\hd_{T_{\red}\times S}\E_{\red} \le 1.$ Then $\hd_{T\times S}
\E\le 1$.
\end{lemma}

We do computations as in \cite{Tim1,Tim8} but now the morphism
$\sigma\!\!\!\sigma$ is defined differently. Choose and fix
locally free $\OO_{T\times S}$-resolution of the sheaf $\E$:
\begin{equation}\label{lfr}
0\to E_1 \to E_0 \to \E \to 0.
\end{equation}
 Apply inverse image $\sigma \!\!\!
\sigma^{\ast}$ to the dual sequence of (\ref{lfr}):
\begin{eqnarray}\xymatrix{
\sigma \!\!\! \sigma^{\ast}\E^{\vee}\ar[r]&\sigma \!\!\!
\sigma^{\ast}E_0^{\vee}\ar[r]
&\sigma \!\!\! \sigma^{\ast}\WW \ar[r]&0, } \label{cc} \nonumber\\
\label{c}\xymatrix{ \sigma \!\!\! \sigma^{\ast}\WW\ar[r]&\sigma
\!\!\! \sigma^{\ast}E_1^{\vee}\ar[r] &\sigma \!\!\!
\sigma^{\ast}\EExt^1 (\E, \OO_{\Sigma})\ar[r]&0.}
\end{eqnarray}
The symbol  $\WW$ stands for the sheaf  $$\ker(E_1^{\vee} \to
\EExt^1(\E, \OO_{\Sigma}))=\coker(\E^{\vee} \to E_0^{\vee}).$$

In  (\ref{c}) denote $\NN:=\ker(\sigma \!\!\!
\sigma^{\ast}E_1^{\vee}\to \sigma \!\!\! \sigma^{\ast}\EExt^1 (\E,
\OO_{\Sigma}))$. The sheaf   $\FFitt^0(\sigma \!\!\!
\sigma^{\ast}\EExt^1(\E, \OO_{\Sigma}))$ is invertible by
functorial property of  $\FFitt$:
\begin{eqnarray*}\FFitt ^0(\sigma \!\!\! \sigma
^{\ast}\EExt^1_{\OO_{\Sigma}}(\E,\OO_{\Sigma}))= (\sigma \!\!\!
\sigma^{-1}\FFitt
^0(\EExt^1_{\OO_{\Sigma}}(\E,\OO_{\Sigma})))\cdot
\OO_{\widehat{\Sigma}}= (\sigma \!\!\! \sigma^{-1}\I)\cdot
\OO_{\widehat{\Sigma}}\\
=(\sigma\!\!\!\sigma^{-1}
i_0^{-1}\I')\cdot \OO_{\widetilde \Sigma}=(\widetilde i_0^{-1}
{\sigma\!\!\! \sigma'}^{-1}\I')\cdot \OO_{\widetilde
\Sigma}=\widetilde i_0^{\ast} ({\sigma\!\!\!\sigma'}^{-1}\I' \cdot
\OO_{\widetilde \Sigma'}).
\end{eqnarray*}
Here we take into account that $\I'=pr_1^{\ast} \I + (t)$ where
$\A^1=\Spec k[t]$, $pr_1: \Sigma \times \A^1\to \Sigma$ is the
natural projection, and the closed immersion $\widetilde i_0:
\widetilde \Sigma \hookrightarrow \widehat \Sigma'$ is fixed by
fibred square
\begin{equation*}
\xymatrix{\widehat \Sigma' \ar[r]^{\sigma\!\!\!\sigma'}& \Sigma
\times \A^1\\
\widetilde \Sigma \ar@{^(->}[u]^{\widetilde i_0}
\ar[r]^{\sigma\!\!\!\sigma}&\Sigma \ar@{^(->}[u]_{i_0}}
\end{equation*}


\begin{lemma}\cite[Lemma 2]{Tim8} Let  $X$ be Noetherian scheme such that its reduction
$X_{\red}$ is irreducible, $\FF$ nonzero coherent $\OO_X$-sheaf
supported on a subscheme of codimension $\ge 1$. Then the sheaf of
0-th Fitting ideals $\FFitt^0(\FF)$ is invertible $\OO_X$-sheaf if
and only if $\FF$ has homological dimension equal to 1: $\hd_X
\FF=1.$
\end{lemma}

Applying the lemma we conclude that    $\hd \sigma \!\!\! \sigma
^{\ast}\EExt^1_{\OO_{\Sigma}}(\E,\OO_{\Sigma})=1.$

Hence the sheaf  $\NN=\ker(\sigma \!\!\!
\sigma^{\ast}E_1^{\vee}\to \sigma \!\!\! \sigma
^{\ast}\EExt^1_{\OO_{\Sigma}}(\E,\OO_{\Sigma})) $ is locally free.
Then there is a morphism of locally free sheaves $\sigma\!\!\!
\sigma^{\ast}E_0^{\vee} \to \NN.$ Let  $Q$ be a sheaf of
$\OO_{\Sigma}$-modules which factors the morphism
 $E_0^{\vee} \to
E_1^{\vee}$ into the composite of epimorphism and monomorphism. By
the definition of the sheaf $\NN$ it also factors the morphism
$\sigma\!\!\!\sigma^{\ast}Q \to
\sigma\!\!\!\sigma^{\ast}E_1^{\vee}$ in the composite of
epimorphism and monomorphism and
$\sigma\!\!\!\sigma^{\ast}E_0^{\vee}\to
\sigma\!\!\!\sigma^{\ast}Q$ is an epimorphism. From this we
conclude that the composite
$\sigma\!\!\!\sigma^{\ast}E_0^{\vee}\to
\sigma\!\!\!\sigma^{\ast}Q\to \NN$ is an epimorphism of locally
free sheaves. Then its kernel is also locally free sheaf. Now set
$\widetilde \E:=\ker(\sigma\!\!\! \sigma^{\ast}E_0^{\vee} \to
\NN)^{\vee}$. Consequently we have an exact triple of locally free
$\OO_{\widehat \Sigma}$-modules $$ 0\to \widetilde \E^{\vee} \to
\sigma\!\!\!\sigma^{\ast}E_0^{\vee} \to \NN \to 0.
$$ Its dual is also exact.

Now there is a commutative diagram with exact rows
\begin{equation}\label{depi}\xymatrix{0\ar[r]& \NN^{\vee}\ar[r]
&\sigma\!\!\!\sigma^{\ast}E_0 \ar[r]&\widetilde \E \ar[r]&0\\
& \sigma\!\!\!\sigma^{\ast}E_1 \ar[u]
\ar[r]&\sigma\!\!\!\sigma^{\ast}E_0 \ar@{=}[u] \ar[r]&
\sigma\!\!\!\sigma^{\ast}\E \ar[u] \ar[r]&0}
\end{equation}
where the right vertical arrow is an epimorphism.

\begin{remark} Since  $\widetilde \E$ is locally free as
 $\OO_{\widetilde
\Sigma}$-module and $\OO_{\widetilde \Sigma}$ is $\OO_T$-flat
then  $\widetilde \E$ is also flat over $T$.
\end{remark}

The epimorphism
\begin{equation}\label{epi}
\sigma\!\!\!\sigma^{\ast} \E \twoheadrightarrow \widetilde
\E\end{equation} induced by the right vertical arrow in
(\ref{depi}), provides quasi-ideality on closed fibres of the
morphism  $\pi$.


The transformation of families we constructed, has a form $$(T,
\L, \E) \mapsto (\pi: \widetilde \Sigma \to T, \widetilde \L,
\widetilde \E)$$ and is defined by the commutative diagram
\begin{equation*}\xymatrix{ T \ar@{=}[d]\ar@{|->}[r]& \{(T, \L,
\E)\} \ar[d]\\
 T \ar@{|->}[r]& \{(\pi: \widetilde \Sigma \to
 T, \widetilde \L, \widetilde \E)\} }
\end{equation*}
 The right vertical arrow is the map of sets.
Their elements are families of objects to be parametrized. The map
is determined by the procedure of resolution as it developed in
this section.
\begin{remark}  The
transformation as it is constructed now  defines a morphism of
functors.
\end{remark}

\section{Pairs-to-GM transformation}


Further we show that there is a morphism of the nonreduced moduli
functor of admiss\-ible semistable pairs to the nonreduced
Gieseker -- Maruyama moduli functor. Namely, for any scheme $T$ we
build up a correspondence $((\pi:\widetilde \Sigma \to T,
\widetilde \L), \widetilde \E)\mapsto (\L, \E)$. It leads to
 a set mapping
 $(\{((\pi:\widetilde \Sigma \to T, \widetilde
\L), \widetilde \E)\}/\sim)\to (\{\L,\E\}/\sim)$. This means that
the family of semistable coherent torsion-free sheaves $\E$ with
the same base $T$ can be constructed by any family
$((\pi:\widetilde \Sigma \to T, \widetilde \L), \widetilde \E)$ of
admissible semistable pairs which is birationally trivial and flat
over $T$.

First we construct a $T$-morphism $\phi: \widetilde \Sigma \to
T\times S$.
Since the family $\pi: \widetilde  \Sigma \to T$ is birationally
trivial there is a fixed isomorphism $\phi_0: \widetilde \Sigma_0
\stackrel{\sim}{\to}\Sigma_0$ of maximal open subschemes
$\widetilde \Sigma_0 \subset \widetilde \Sigma$ and
$\Sigma_0 \subset T\times S$. Define an invertible
$\OO_{T\times S}$-sheaf $\L$ by the equality
$$\L(U):= \widetilde \L(\phi_0^{-1}(U \cap \Sigma_0)).$$ Identifying
$\widetilde \Sigma_0$ with $\Sigma_0$ by the isomorphism $\phi_0$
one comes to the conclusion that sheaves $\L|_{\Sigma_0}$ and
$\widetilde \L|_{\widetilde \Sigma_0}$ are also isomorphic.
\begin{proposition} \label{polariz} For any closed point $t\in T$ and for any open $V\subset S$
$$
\L\otimes (k_t\boxtimes \OO_S)(V)=\widetilde L_t(\sigma^{-1}(V) \cap \widetilde \Sigma_0).
$$
In particular, $\L\otimes (k_t \boxtimes \OO_S)=L.$
\end{proposition}
\begin{proof} The restriction $\L|_{t\times S}$ is the sheaf associated
to a presheaf $$V \mapsto \L(U) \otimes _{\OO_{T\times S}(U)}
(k_t\boxtimes \OO_S)(U\cap t\times S)$$
 for any open $U\subset T\times S$ such that $U\cap (t\times S)=V.$
Since $\codim T\times S \setminus \Sigma_0 \ge 2,$
$$\OO_{
T\times S}(U)=\OO_{T\times S}(U \cap \Sigma_0)$$ and
$$(k_t \boxtimes \OO_S)(U\cap t\times S)=
(k_t \boxtimes \OO_S)(U\cap \Sigma_0\cap t\times S)
=\OO_{\widetilde S_t}(\phi_0^{-1}(U\cap \Sigma_0)).$$
Hence $\L|_{t\times S}$ is associated to the presheaf $$
V\mapsto \widetilde \L(\phi_0^{-1}(U\cap \Sigma_0))
\otimes_{\OO_{\widetilde \Sigma}(\phi_0^{-1}(U\cap \Sigma_0))}
\OO_{\pi^{-1}(t)}(\phi_0^{-1}(U\cap \Sigma_0)\cap \pi^{-1}(t)),
$$ or, equivalently,
$$V\mapsto
\widetilde L_t (\phi_0^{-1}(U\cap \Sigma_0)\cap \widetilde S_t)
=L(U\cap \Sigma_0\cap t\times S)=L(U\cap t\times S).$$
We keep in mind that $\phi_0^{-1}(U\cap \Sigma_0)\cap
\widetilde S_t=\sigma^{-1}(V)\cap \widetilde \Sigma_0$.\end{proof}

Define a sheaf $\L'$ by the correspondence
$U \mapsto \widetilde \L(U\cap \widetilde \Sigma_0)$ for any open
$U\subset \widetilde \Sigma.$ It carries a natural structure of
invertible $\OO_{\widetilde \Sigma}$-module. This structure is
induced by the commutative diagram
$$\xymatrix{\OO_{\widetilde \Sigma}(U)\times \L'(U)
\ar[d]_{res} \ar[r]&\L'(U) \ar@{=}[d]\\
\OO_{\widetilde \Sigma}(U\cap \widetilde \Sigma_0)\times
\widetilde \L(U\cap \widetilde \Sigma_0)\ar[r]
& \widetilde \L(U\cap \widetilde \Sigma_0)}
$$
where vertical arrow is induced by the natural restriction map
in $\OO_{\widetilde \Sigma}$.
Compare direct images $p_{\ast}\L$ and $\pi_{\ast} \L'$; for
any open $V\subset T$
$$p_{\ast}\L (V)=\L(p^{-1}V)=\L(p^{-1}V \cap \Sigma_0)
$$
By the definition of $\L'$
$$
\L(p^{-1}V \cap \Sigma_0)=\widetilde \L(\pi^{-1}V \cap
\widetilde \Sigma_0)=\L'(\pi^{-1}V)=\pi_{\ast}\L'(V).
$$
Now $\pi_{\ast} \L'=p_{\ast} \L$.

The invertible sheaf $\L'$ induces a morphism
$\phi':\widetilde \Sigma \to \P(\pi_{\ast}\L')^\vee$
which includes into the commutative diagram of $T$-schemes
$$\xymatrix{\widetilde\Sigma \ar[rr]^{\phi'}&&\P(\pi_{\ast}\L')^{\vee} \ar@{=}[dd] \\
\widetilde \Sigma_0=\Sigma_0 \ar@{^(->}[u] \ar@{_(->}[d]\\
T\times S \ar@{^(->}[rr]^{i_{\L}}&& \P(p_{\ast}\L)^{\vee}}
$$
where $i_{\L}$ is a closed immersion induced by $\L$ and
$\phi'|_{\widetilde \Sigma_0}$ is also immersion. From now
we identify $\P(\pi_{\ast}\L')^{\vee}$ and $\P(p_{\ast}\L)^{\vee}$
and use common notation $\P$ for these bundles.
Formation of scheme closures of images of $\widetilde \Sigma_0$
and $\Sigma_0$ in $\P$ leads to
$\overline{\phi'(\widetilde \Sigma_0)}=\overline{i_\L(T\times S)}
=T\times S.$ Also by the definition of the sheaf $\L'$ for
any open $U\subset \widetilde \Sigma$ and $V\subset T\times S$
such then $U\cap \widetilde \Sigma_0 \cong V\cap \Sigma_0$ the
following chain of equalities holds: \begin{equation} \label{corrsec}
\L'(U)= \L'(U\cap \widetilde \Sigma_0)=
\L(V\cap \Sigma_0)=\L(V).\end{equation}
Now for a moment we suppose that $T$ is affine: $T=\Spec A$ for
some commutative algebra $A$, $\P=\Proj A[x_0:\dots :x_N]$ where
$x_0,\dots,x_N\in H^0(\P, \OO_\P (1))$ generate $\OO_\P (1)$.
Images ${\phi'}^{\ast} x_i=s_i',$ $i=0,\dots, N$, generate $\L'$
along $\widetilde \Sigma_0$ and they are not obliged to generate
$\L'$ along the whole of $\widetilde \Sigma.$ Images
$i_\L^{\ast} x_i=s_i$, $i=0,\dots, N$, generate $\L$ along the whole of $T\times S$
and provide that $i_\L$ is closed immersion.

We pass to standard affine covering by
$\P_i=\Spec A[x_0,\dots, \hat x_i,\dots, x_N],$ $i=0, \dots, N$,
and $\hat{\phantom q}$ means omitting the symbol below. Denote
$(i_\L(T\times S))_i:=i_\L(T\times S)\cap \P_i$ and
$(\phi'(\widetilde \Sigma))_i:=\phi'(\widetilde \Sigma) \cap \P_i.$
Set also $(T\times S)_i:=i_\L ^{-1}(i_\L(T\times S))_i$ and
$\widetilde \Sigma_i:={\phi'}^{-1}(\phi'(\widetilde \Sigma))_i$.
Now we have mappings $$
A[x_0,\dots, \hat x_i,\dots, x_N] \rightarrow \Gamma(\widetilde \Sigma_i, \L'): x_j \mapsto s'_j
$$
and
$$
A[x_0,\dots, \hat x_i,\dots, x_N] \twoheadrightarrow \Gamma((T\times S)_i, \L): x_j \mapsto s_j
$$
which fit into triangular diagram
\begin{equation}\label{diasec}\xymatrix{A[x_0,\dots, \hat x_i,\dots, x_N]
\ar@{->>}[rd] \ar@{->>}[r]& \Gamma((T\times S)_i, \L) \ar@{=}[d]\\
&\Gamma(\widetilde \Sigma_i, \L')}\end{equation} where the vertical sign of
equality means bijection (\ref{corrsec}). Commutativity of (\ref{diasec})
implies that $\phi'$ factors through $i_\L(T\times S)$, i.e. $\phi'(\widetilde \Sigma) = i_\L (T\times S).$

Now identifying $i_\L(T\times S)$ with $T\times S$ by means of
obvious isomorphism we arrive to the $T$-morphism
$$ \phi: \widetilde \Sigma \to T\times S.
$$
It coincides with $\phi_0: \widetilde \Sigma_0
\stackrel{\sim}{\to} \Sigma_0$ when restricted to $\widetilde \Sigma_0.$

For $n>0$ consider an invertible $\OO_{S\times T}$-sheaf
$U \mapsto \widetilde \L^n(\phi_0^{-1}(U \cap \Sigma_0))$. It
coincides with $\L^n$ on $\Sigma_0$ and hence it coincides with
it in total.

Now there is a commutative triangle
\begin{equation}\xymatrix{\widetilde \Sigma \ar[r]^\phi
\ar[rd]_\pi &T\times S \ar[d]^p\\&T}
\end{equation}

Firstly note that $T$ contains at least one closed point,
say $t\in T$; let $\widetilde S_t=\pi^{-1}(t)$ be
the corresponding closed fibre and
$\widetilde L_t=\widetilde \L|_{\widetilde S_t}$ and
$\widetilde E_t=\widetilde \E|_{\widetilde S_t}$
restrictions of sheaves to it. By the definition of admissible
scheme there is a canonical morphism $\sigma: \widetilde S_t \to S$.
Then $(\sigma_{\ast} \widetilde L_t)^{\vee \vee}=L.$

Secondly, the family $\pi: \widetilde \Sigma \to T$ is birationally
trivial, i.e. there exist isomorphic open subschemes
$\widetilde \Sigma_0 \subset \widetilde \Sigma$ and
$\Sigma_0 \subset T\times S$. Note that the "boundary"
$\Delta=S\times T \setminus \Sigma_0$  has codimension $\ge 2$
and that for any closed point $t\in T$ $\codim \Delta\cap (t\times S)\ge 2.$

Thirdly, the morphism of multiplication of sections $$
(\sigma_{\ast} \widetilde L_t)^n \to \sigma_{\ast} \widetilde L_t^n
$$
induces the morphism of reflexive hulls
$$
((\sigma_{\ast} \widetilde L_t)^n )^{\vee \vee}\to (\sigma_{\ast} \widetilde L_t^n)^{\vee \vee}
$$
which are locally
free sheaves on a surface and coincide apart from a collection of
points. Hence they are equal.
Also the sheaf $((\sigma_{\ast} \widetilde L_t)^{\vee \vee})^n=
L^n$ coincides with them by the analogous reason. Then for all $n>0$
$$
((\sigma_{\ast} \widetilde L_t)^n )^{\vee \vee} =L^n.
$$

Now take a product $\A^1 \times S$,  $\A^1 =\Spec k[u]$.
Let $I\subset \OO_S$ be the sheaf of ideals such that
$\widetilde S_t =\Proj \bigoplus_{s\ge 0}(I[u]+(u))^s/(u^{s+1})$ and
the blowing up $Bl_{\II}\A^1 \times S$ in
the sheaf of ideals $\II=(u)+I[u]$ corresponding to the subscheme
with ideal $I$ in the zero-fibre $0\times S$. If
$\sigma\!\!\!\sigma: Bl_{\II}\A^1 \times S \to {\A^1
\times S}$ is the blowup morphism then  a
zero-fibre $\widehat S_0$ of the composite
$$pr_1\circ \sigma\!\!\!\sigma: Bl_{\II}\A^1 \times S \to
{\Spec k[u] \times S} \to \Spec k[u]$$ is isomorphic to
$\widetilde S_t$. Other closed fibres are isomorphic to $S$.
Since this composite is flat morphism, the invertible sheaf
$\widehat \L=\sigma \!\!\!\sigma^{\ast}(\OO_{\A^1}\boxtimes L)
\otimes \sigma \!\!\! \sigma^{-1}\II \cdot \OO_{Bl_{\II}\A^1
\times S}$ is flat over $\A^1$. Now
$\widetilde L_t=\widehat \L|_{\widehat S_0}$, and
for $n\gg 0$ one has
$$h^0(\widetilde S_t, \widetilde L_t^n)=h^0(S,L^n).$$

\begin{proposition} There are morphisms of $\OO_{T\times S}$
-sheaves $$
\phi_{\ast} \widetilde \L^n \to \L^n$$
for all $n>0$.
\end{proposition}
\begin{proof}
For any open $U\in T\times S$ and any $n>0$ there is a restriction
map of sections $res: (\phi_{\ast} \widetilde \L^n)(U)\to
(\phi_{\ast} \widetilde \L^n)(U\cap \Sigma_0).$ Denoting as usually
the preimage $\phi^{-1}(\Sigma_0)$ by $\widetilde \Sigma_0$ (recall
that $\phi|_{\widetilde \Sigma_0}=\phi_0$ is an isomorphism) one arrives
to the chain of equalities:
$$
(\phi_{\ast} \widetilde \L^n)(U\cap \Sigma_0)=
\widetilde \L^n (\phi^{-1}(U\cap \Sigma_0))=
 \L^n (U).
$$
\end{proof}
Applying $p_{\ast}$ yields in
\begin{corollary} For $n>0$ morphisms $
\phi_{\ast} \widetilde \L^n \to \L^n$ induce isomorphisms of $\OO_T$-sheaves
$$\pi_{\ast} \widetilde \L^n \stackrel{\sim}{\to} p_{\ast}\L^n.
$$
\end{corollary}
\begin{proof} Both sheaves $\pi_{\ast} \widetilde \L^n$ and  $p_{\ast}\L^n$ are
locally free and have equal ranks. Passing to fibrewise
consideration one gets $\pi_{\ast} \widetilde \L^n \otimes k_t
\to p_{\ast}\L^n \otimes k_t$ or, equivalently,
$H^0(\widetilde S_t, \widetilde L^n_t)\to H^0(t\times S, L^n)$.
This map is an isomorphism and hence
$\pi_{\ast} \widetilde \L^n \stackrel{\sim}{\to} p_{\ast}\L^n$.\end{proof}

We will need sheaves $$\widetilde \V_m=\pi_{\ast}(\widetilde \E
\otimes \widetilde \L^m)$$ for $m\gg 0$ such that $\widetilde \V_m$
are locally free of rank $rp(m)$ and $\widetilde \E \otimes \L^m$ are fibrewise
globally generated in such sense that the canonical morphisms
$$\pi^{\ast} \widetilde \V_m \to \widetilde
\E \otimes \widetilde \L^m$$ are surjective for those $m$'s.

Let also for $m\gg 0$ $$\EE_m=\phi_{\ast} (\widetilde \E
\otimes \widetilde \L^m),$$ now $$p_{\ast} \EE_m =
p_{\ast} \phi_{\ast}(\widetilde \E\otimes \widetilde \L^m)=
\pi_{\ast}(\widetilde \E\otimes \widetilde \L^m)=\widetilde \V_m.$$

We intend to confirm that sheaves $p_{\ast} (\EE_m \otimes \L^n)$
are locally free of rank $rp(m+n)$ for all $m,n\gg 0$. This
implies $T$-flatness of $\EE_m$.

To proceed further we need morphisms $\widetilde \L^n \to \phi^{\ast} \L^n,$ $n>0.$

\begin{proposition} \label{injinv} For all $n>0$ there are injective morphisms
$\iota_n: \widetilde \L^n \to \phi^{\ast}\L^n$
of invertible $\OO_{\widetilde \Sigma}$-sheaves.
\end{proposition}
\begin{proof}
For $n>0$ and for any open $U\subset \widetilde \Sigma$ there
is a restriction map on sections
$$\widetilde \L^n(U) \stackrel{res}{\longrightarrow}
\widetilde \L^n(U\cap \widetilde \Sigma_0)=
\L^n(\phi_0(U\cap \widetilde \Sigma_0))=
\L^n(\phi_0(U) \cap \Sigma_0)=\L^n(\phi(U)).
$$
Since $\phi$ is projective and hence takes closed subsets to
closed subsets (resp., open to open), this implies the sheaf
morphism $\widetilde \L^n \to \phi^{-1}\L^n$. Combining it with
multiplication by unity section $1\in \OO_{\widetilde \Sigma}(U)$
leads to the morphism $\iota_n: \widetilde \L^n \to \phi^{\ast} \L^n$ of
invertible $\OO_{\widetilde \Sigma}$-modules.
\end{proof}
\begin{remark} By the definition of the invertible
$\OO_{\widetilde \Sigma}$-sheaf $\L'$ it follows from the proof
done that there are injective morphisms of invertible
$\OO_{\widetilde \Sigma}$-modules $\widetilde \L^n \to {\L'}^n.$
\end{remark}
\begin{proposition}\label{fltnss} $\EE_m$ are $T$-flat for $m\gg 0.$
\end{proposition}
\begin{proof} Consider the morphism of multiplication of sections
$$
p_{\ast} \phi_{\ast} (\widetilde \E \otimes \widetilde \L^m)
\otimes p_{\ast} \L^n
\to p_{\ast}(\phi_{\ast}(\widetilde \E \otimes \widetilde \L^m)
\otimes \L^n)$$
which is surjective for $m,n\gg 0$. By the projection formula
$$p_{\ast}(\phi_{\ast}(\widetilde \E \otimes \widetilde \L^m)
\otimes \L^n)=
p_{\ast}(\phi_{\ast}(\widetilde \E \otimes \widetilde \L^m\otimes
\phi^{\ast}\L^n).
$$
Also for the projection $\pi$ we have another morphism of
multiplication of sections
$$\pi_{\ast}(\widetilde \E \otimes \widetilde \L^m)
\otimes \pi_{\ast} \widetilde \L^n
\to \pi_{\ast}(\widetilde \E \otimes \widetilde \L^{m+n})
$$
Injective $\OO_{\widetilde \Sigma}$-morphism
$\widetilde \L^n \hookrightarrow \phi^{\ast}\L^n$ after
tensoring by $\widetilde \E \otimes \widetilde \L^m$ and
applying $\pi_\ast$ leads to
$$ \pi_{\ast}(\widetilde \E \otimes \widetilde \L^{m+n}) \hookrightarrow
\pi_{\ast}(\widetilde \E \otimes \widetilde \L^m\otimes \phi^{\ast}\L^n)
$$
Taking into account the isomorphism $p_{\ast}\L^n=
\pi_{\ast}\widetilde \L^n$ and Proposition \ref{injinv}
we gather these mappings into the commutative diagram
\begin{equation} \xymatrix{
p_{\ast} \phi_{\ast}(\widetilde \E \otimes
\widetilde \L^m)\otimes p_{\ast}\L^n \ar@{=}[d] \ar@{->>}[r]&
p_{\ast}(\phi_{\ast}(\widetilde \E \otimes \widetilde \L^m \otimes
\phi^{\ast}\L^n))\\
\pi_{\ast}(\widetilde \E \otimes \widetilde L^m)\otimes \pi_{\ast}\widetilde \L^n
\ar@{->>}[r]& \pi_{\ast}(\widetilde \E\otimes \widetilde \L^{m+n}) \ar@{^(->}[u]}\end{equation}
By commutativity of this diagram we conclude that
\begin{equation} \label{eqtydir}\pi_{\ast}(\widetilde \E\otimes \widetilde \L^{m+n})=
p_{\ast}(\phi_{\ast}(\widetilde \E \otimes \widetilde \L^m)
\otimes \L^n))\end{equation} or, in our notation,
$p_{\ast}(\EE_m\otimes \L^n)=\widetilde \V_{m+n}$ for $m,n \gg 0$. This
guarantees that $\EE_m$ are $T$-flat for $m\gg 0.$\end{proof}

We intend to confirm that $\EE_m\otimes \L^{-m}$ are families of
semistable sheaves on $S$ as we need. First we prove the following
\begin{proposition} $\EE_{m+n}=\EE_m \otimes \L^n$ for any $m\gg 0,n>0$.
\end{proposition}
\begin{proof}
By the definition if sheaves
$\EE_m=\phi_{\ast}(\widetilde \E \otimes \widetilde \L^m)$ for
any $m>0$ there
is an injective $\OO_{T\times S}$-morphism
$$\varepsilon_{m+n}: \EE_m \hookrightarrow \EE_{m+n}
$$
induced locally by multiplication by generator of $\widetilde \L^n$.

Consider the sheaf inclusion $\iota_n:\widetilde \L^n \hookrightarrow
\phi^{\ast}\L^n$ valid for any $n>0$. Tensoring by locally free
$\OO_{\widetilde \Sigma}$- module
$\widetilde \E \otimes \widetilde \L^m$, formation of
direct image under $\phi$ and projection formula yield in the
inclusion $$i_{m,n}:\EE_{m+n} \hookrightarrow \EE_m \otimes \L^n.
$$
Both sheaves $\EE_m\otimes \L^n$ and $\EE_{m+n}$  become normal if
restricted to $T_{\red}\times S$
and coincide apart from their singular locus $T_\red \times S
\setminus \Sigma_{0\red}$ which has codimension $\ge 2.$ Hence
they coincide along the whole of $T_\red \times S.$ Let $\TT$ be
the cokernel of $i_{m,n}$; since $\EE_{m+1}|_{\Sigma_0}=
\EE_{m}\otimes \L|_{\Sigma_0}$ that $\Supp \TT \subset T\times S
\setminus \Sigma_0.$ If $\Supp \TT \ne \emptyset$ it contains at
least one closed point $t\times s$ and $\TT_{t\times s}\ne 0$. Now
$t\times s\in T_\red$ but $\TT \boxtimes \OO_{T_\red}=0$. This
contradiction leads us to the conclusion that $\Supp \TT =\emptyset$
and $\TT=0$ what proves the proposition.
\end{proof}

 We can introduce the goal sheaf of our construction
$$\E:=\EE_m \otimes \L^{-m}.$$
By the proposition proved this definition is independent of $m$
at least in case when $m\gg 0.$ The sheaf $\E$ is $T$-flat.

\begin{proposition}\label{hpol} The sheaf $\E$ with respect to the invertible
sheaf $\L$ has fibrewise Hilbert polynomial equal to $rp(n),$ i.e.
for $n\gg 0$
$$\rank p_{\ast} (\E \otimes \L^n) =rp(n).$$
\end{proposition}
\begin{proof}
For $n\gg m\gg 0$ by (\ref{eqtydir}) we have chain of equalities $p_{\ast} (\E \otimes \L^n)=
p_{\ast}(\EE_m \otimes \L^{n-m})=
p_{\ast}(\phi_{\ast}(\widetilde \E \otimes \widetilde \L^m)
\otimes \L^{n-m})= \pi_{\ast}(\widetilde \E \otimes
\widetilde \L^n)$. The recent sheaf of the chain has rank equal to
$rp(n).$
\end{proof}

\begin{proposition} For any closed point $t\in T$ the sheaf $$
E_t:=\E|_{t\times S}
$$
is torsion-free and Gieseker-semistable with respect to
$$
L_t:=\L|_{t\times S}\cong L.
$$

\end{proposition}
\begin{proof}
The isomorphism $\L|_{t\times S}\cong L$ is the subject of
Proposition \ref{polariz}. Now for $E_t$ one has
$E_t=\E|_{t\times S}=(\EE_m \otimes \L^{-m})|_{t\times S}=
\EE_m|_{t\times S} \otimes L^{-m}=\phi_{\ast} (\widetilde \E
\otimes \widetilde \L^m)|_{t\times S}\otimes L^{-m}.$ Denoting by
$i_t: t\times S \hookrightarrow T\times S$
and $\widetilde i_t: \widetilde S_t \hookrightarrow \widetilde
\Sigma$ morphisms of closed immersions of  fibres
 we come to
$\phi_{\ast} (\widetilde \E
\otimes \widetilde \L^m)|_{t\times S}\otimes L^{-m} =(i_t^{\ast}\phi_{\ast} (\widetilde \E
\otimes \widetilde \L^m))\otimes L^{-m}$ and base change morphism
\begin{equation} \label{bchfib}\beta_t:i_t^{\ast}\phi_{\ast} (\widetilde \E
\otimes \widetilde \L^m) \to \sigma_{t \ast}\widetilde i_t^{\ast}(\widetilde \E \otimes \widetilde \L^m)
\end{equation}
 in the fibred square
$$
\xymatrix{t\times S \ar@{^(->}[r]^{i_t}&T\times S\\
\widetilde S_t \ar[u]^{\sigma_t} \ar@{^(->}[r]^{\widetilde i_t}& \widetilde \Sigma \ar[u]_\phi
}
$$
Quasi-ideality of sheaf $\widetilde \E_t=
\widetilde \E|_{\widetilde S_t}$ validates the following lemma to be proven below.
\begin{lemma}\label{torsfree} The sheaf $\sigma_{t\ast} \widetilde E_t$ is torsion-free.
\end{lemma}
Both sheaves in (\ref{bchfib}) coincide along $(t\times S) \cap
\Sigma_0.$ Now consider corresponding map of global sections:
$$H^0(\beta_t): H^0(t\times S, i_t^{\ast}\phi_{\ast} (\widetilde \E
\otimes \widetilde \L^m)) \to H^0(t\times S,\sigma_{t \ast}
\widetilde i_t^{\ast}(\widetilde \E \otimes \widetilde \L^m)).$$ It is injective.
Left hand side takes the view
$$H^0(t\times S, i_t^{\ast}\phi_{\ast} (\widetilde \E
\otimes \widetilde \L^m))\otimes k_t=
i_t^\ast p_{\ast}\phi_{\ast} (\widetilde \E
\otimes \widetilde \L^m)=k_t^{\oplus rp(m)}.$$
Also for right hand side one has $$H^0(t\times S,\sigma_{t \ast}
\widetilde i_t^{\ast}(\widetilde \E \otimes \widetilde \L^m)) \otimes k_t=
H^0(\widetilde S_t, \widetilde E_t \otimes \widetilde L_t^m)\otimes k_t =
k_t^{\oplus rp(m)}.$$ This implies that $H^0(\beta_t)$ is bijective
and there is a commutative diagram
$$
\xymatrix{H^0(t\times S, \sigma_{t\ast}
(\widetilde E_t\otimes \widetilde L_t^m))\otimes \OO_S
\ar@{=}[d]_{H^0(\beta_t)} \ar@{->>}[r]&
\sigma_{t\ast}(\widetilde E_t\otimes \widetilde L_t^m)\\
i_t^{\ast} p^{\ast}\widetilde \V_m \ar@{->>}[r]&
i_t^{\ast}\EE_m \ar[u]_{\beta_t}}
$$
It implies that $\beta_t$ is surjective. Since
$\ker H^0(\beta_t)=H^0(t\times S,\ker \beta_t)=0,$
then $\beta_t$ is isomorphic.

Now take a subsheaf $F_t \subset E_t$. Now for $m \gg 0$ there
is a commutative diagram
$$\xymatrix{H^0(t\times S, E_t\otimes L^m) \otimes \OO_S \ar@{->>}[r]& E_t \otimes L^m\\
H^0(t\times S, F_t\otimes L^m) \otimes \OO_S \ar@{^(->}[u] \ar@{->>}[r]& F_t \otimes L^m \ar@{^(->}[u]}
$$
The isomorphism $E_t\otimes L^m=
\sigma_{t\ast}(\widetilde E_t \otimes \widetilde L_t^m)$ proven
above fixes bijection on global sections $H^0(t\times S, E_t \otimes L^m)\simeq
H^0(t\times S, \sigma_{t\ast}(\widetilde E_t \otimes \widetilde L_t^m))=
H^0(\widetilde S_t, \widetilde E_t \otimes \widetilde L_t^m).$
Let $\widetilde V_t \subset H^0(\widetilde S_t, \widetilde E_t \otimes
\widetilde L_t^m)$ be the subspace corresponding to $H^0(t\times S, F_t \otimes L^m)\subset
H^0(t\times S, E_t \otimes L^m)$ under this bijection.
Now one has a commutative diagram
$$
\xymatrix{
H^0(\widetilde S_t, \widetilde E_t \otimes \widetilde L_t^m)
\otimes \OO_{\widetilde S_t} \ar@{->>}[r]& \widetilde E_t \otimes
\widetilde L_t^m\\
\widetilde V_t \otimes \OO_{\widetilde S_t} \ar@{^(->}[u]_\Upsilon
\ar@{->>}[r]^{\varepsilon}& \widetilde F_t \otimes \widetilde L_t
\ar@{^(->}[u]}
$$
where $\widetilde F_t \otimes \widetilde L_t^m \subset
\widetilde E_t \otimes \widetilde L_t^m$ is defines as a subsheaf
generated by the subspace $\widetilde V_t$ by means of the morphism
$\varepsilon$. The associated map of global sections
$$
H^0(\varepsilon): \widetilde V_t \to H^0(\widetilde S_t,
\widetilde F_t \otimes \widetilde L_t^m)$$  includes into
the commutative triangle
$$
\xymatrix{\widetilde V_t \ar@{_(->}[rd] \ar[r]^<<<<<<{H^0(\varepsilon)} &H^0(\widetilde S_t,
\widetilde F_t \otimes \widetilde L_t^m) \ar@{^(->}[d]\\
& H^0(\widetilde S_t,
\widetilde E_t \otimes \widetilde L_t^m) }
$$
what implies that $H^0(\varepsilon)$ is injective. Since each
section from $H^0(\widetilde S_t,
\widetilde F_t \otimes \widetilde L_t^m)$ corresponds to a
section in $H^0(t\times S, F_t \times L^m)\subset
H^0(t\times S, E_t \otimes L^m)$ then $H^0(\varepsilon)$ is
surjective. Hence $h^0(t\times S, F_t \otimes L^m)=
h^0(\widetilde S_t, \widetilde F_t \otimes \widetilde L_t^m)$
for all $m\gg 0$ and stability (resp., semistability) for
$\widetilde E_t$ implies stability (resp., semistability)
for $E_t.$
\end{proof}
\begin{proof}[Proof of Lemma \ref{torsfree}] Since sheaves
$\widetilde E_t$ and $\sigma_{t\ast} \widetilde E_t$ coincide
along identified open subschemes $\widetilde S_t \cap \widetilde
\Sigma_0 \simeq t\times S \cap \Sigma_0$, it is enough to confirm
that there is no torsion subsheaf concentrated on
$t\times S \cap (T\times S \setminus \Sigma_0)$ in
$\sigma_{t\ast} \widetilde E_t$.
Assume that $T\subset \sigma_{t\ast} \widetilde E_t$ is such a
torsion subsheaf i.e. $T \ne 0$ and for any open
$U \subset t\times S \cap \Sigma_0$ $T(U)=0$. Let
$A =\Supp T \subset t\times S$ and $U\subset t\times S$ be such
an open subset that $T(U)\ne 0,$ i.e. $U\cap A \ne \emptyset.$
Now
$$
T(U)\subset \sigma_{t\ast} \widetilde E_t (U)=
\widetilde E_t (\sigma^{-1}U),
$$
and any nonzero section $s\in T(U)$ is supported in $U\cap A$
and comes from the section
$\widetilde s\in \widetilde E_t(\sigma^{-1}U)$ with support in
$\sigma^{-1}(U\cap A).$ This means that $\widetilde s$ is supported
in some additional component $\widetilde S_{t,j}$ of the admissible
scheme $\widetilde S_t$. Hence $\widetilde s \in \tors_j$.
But by quasi-ideality of $\widetilde E_t$ on additional components
$\widetilde S_{t,j}$ of $\widetilde S_t$, $\tors_j=0.$ This implies
that $T=0$.
\end{proof}

\section{Functor isomorphism}
In this section we prove that natural transformations
$\underline \kappa: {\mathfrak f}^{GM}\to {\mathfrak f}$ and
$\underline \tau: {\mathfrak f}\to {\mathfrak f}^{GM}$ are
mutually inverse and hence provide the isomorphism between
the functor of moduli of admissible semistable pairs and
the functor of moduli in the sense of Gieseker and Maruyama.
As a corollary, we get the isomorphism of moduli schemes for
these moduli functors, with no dependence on number and geometry
of their connected components, on reducedness of scheme structure
and on presence of locally free sheaves (respectively, $S$-pairs)
in each component.

For this purpose we perform the proof in two aspects.
\begin{enumerate}
\item{{\it Pointwise.} {\it a)} For any torsion-free semistable
$\OO_S$-sheaf the composite of transform\-at\-ions $$
E \mapsto ((\widetilde S, \widetilde L), \widetilde E) \mapsto E'$$
returns $E'=E.$ \\
{\it b)} Conversely, for any admissible semistable pair
$((\widetilde S, \widetilde L), \widetilde E)$ the composite of
trans\-form\-at\-ions $$
((\widetilde S, \widetilde L), \widetilde E)\mapsto E \mapsto
((\widetilde S', \widetilde L'), \widetilde E')$$
returns $((\widetilde S', \widetilde L'), \widetilde E')=
((\widetilde S, \widetilde L), \widetilde E).$}
\item{{\it Global.} {\it a)} For any family of semistable torsion-free
sheaves $\E$ with base scheme $T$ the composite
$$(\E, \L) \mapsto ((\pi:\widetilde \Sigma \to T, \widetilde \L),
\widetilde \E) \mapsto (\E', \L')
$$ returns such $\E'$ that there is an invertible $\OO_T$-sheaf
$\LL$ with property
$\E'=\E \otimes p^{\ast}\LL.$ \\
{\it b)} Conversely, for any family
$((\pi:\widetilde \Sigma \to T, \widetilde \L), \widetilde \E)$
of admissible semistable pairs with base scheme $T$ the composite
of transformations $$
((\pi:\widetilde \Sigma \to T, \widetilde \L), \widetilde \E)
\mapsto (\E,\L) \mapsto
((\pi':\widetilde \Sigma' \to T, \widetilde \L'), \widetilde \E')
$$ returns family $((\pi':\widetilde \Sigma' \to T, \widetilde \L'),
\widetilde \E')$ with following properties: there is
$T$-isomorphism $\widetilde \Sigma'=\widetilde \Sigma$, $\pi =\pi'$
and there are invertible
$\OO_T$-sheaves $\LL', \LL''$ such that
$\widetilde \L'=\widetilde \L \otimes {\pi'}^\ast \LL'$ and
$\widetilde \E'=\widetilde \E \otimes {\pi'}^{\ast} \LL''.$}
\end{enumerate}

We begin with 2 {\it a)}; it will be specialized to pointwise
version 1 {\it a)} when $T=\Spec k.$

Families of polarizations $\L$ and $\L'$ coincide along the open
subset (locally free locus for sheaves $\E$ and $\E'$) $\Sigma_0$
where $\codim _{T\times S} T\times S \setminus \Sigma_0\ge 2$.
Since $\L$ and $\L'$ are locally free  this implies that $\L=\L'$.


Now consider three locally free $\OO_T$-sheaves of ranks equal to
$rp(m)$:
$$\V_m=p_{\ast}(\E\otimes \L^m),\quad
\widetilde \V_m=\pi_{\ast} (\widetilde \E\otimes \widetilde \L^m),\quad
\V'_m=p_{\ast}(\E'\otimes \L^m).$$
\begin{lemma} $\V_m \cong \widetilde \V_m \cong \V'_m$.
\end{lemma}
\begin{proof} Start with the epimorphism
$\sigma\!\!\!\sigma^{\ast} \E \twoheadrightarrow \widetilde \E.$
Tensoring it by $\widetilde \L^m$ and direct image
$\sigma\!\!\!\sigma_{\ast}$ yield in the morphism of $\OO_T$-sheaves
$$\sigma\!\!\!\sigma_{\ast}(\sigma\!\!\!\sigma^{\ast} \E \otimes
\widetilde \L^m) \rightarrow \sigma\!\!\!\sigma_{\ast}(\widetilde
\E \otimes \widetilde \L^m).
$$
Then the following result will be of use.
\begin{lemma}\label{profor} \cite{Tim9} Let $f: (X,\OO_X) \to (Y, \OO_Y)$ be a morphism of
locally ringed spaces such that $f_{\ast}\OO_X=\OO_Y$, $\EE$
$\OO_Y$-module of finite presentation, $\FF$  $\OO_X$-module. Then
there is a monomorphism $\EE \otimes f_{\ast}\FF \hookrightarrow
f_{\ast}[f^{\ast}\EE \otimes \FF].$
\end{lemma}
Setting $\EE=\E,$ $\FF=\widetilde \L^m$ and $f=\sigma\!\!\!\sigma$
we get
$$
\xymatrix{
\E \otimes \L^m \otimes \sigma\!\!\!\sigma_{\ast}
(\sigma\!\!\!\sigma^{-1} \I\cdot \OO_{\widetilde \Sigma})^m \ar@{^(->}[d] \\
\sigma\!\!\!\sigma_{\ast}(\sigma\!\!\!\sigma^{\ast}\E \otimes \widetilde \L^m)\ar[r]
& \sigma\!\!\!\sigma_{\ast}(\widetilde \E \otimes \widetilde \L^m)}
$$
and after taking direct image $p_{\ast}$
$$
\xymatrix{
p_{\ast}(\E \otimes \L^m \otimes \sigma\!\!\!\sigma_{\ast}
(\sigma\!\!\!\sigma^{-1} \I\cdot \OO_{\widetilde \Sigma})^m )\ar@{^(->}[d]\ar[dr]^{\eta}\ar@{^(->}[r]&p_{\ast}(\E \otimes \L^m) \\
p_{\ast}\sigma\!\!\!\sigma_{\ast}(\sigma\!\!\!\sigma^{\ast}\E \otimes \widetilde \L^m)\ar[r]
& \pi_{\ast}(\widetilde \E \otimes \widetilde \L^m)}
$$
where upper horizontal arrow is natural immersion into reflexive
hull. Since target sheaf of the composite map $\eta$ is reflexive,
$\eta$ factors through $p_{\ast}(\E\otimes \L^m)$ as reflexive hull
of the source. This yields in existence of the morphism of
locally free sheaves of equal ranks
$$
\widetilde \eta:p_{\ast}(\E\otimes \L^m) \to \pi_{\ast}(\widetilde
\E \otimes \widetilde \L^m).
$$
The morphism of sheaves is an isomorphism iff it is stalkwise
isomorphic.
Fix an arbitrary closed point $t\in T$; till the end of this
proof we omit subscript $t$ in notations corresponding to the
sheaves corresponding to $t$: $\E_t=:E,$
$\widetilde E_t=:\widetilde E,$ $\widetilde S_t=:\widetilde S,$
$\sigma_t=:\sigma: \widetilde S \to S.$
There is an epimorphism of $\OO_{\widetilde S}$-modules
$$
\sigma^{\ast} E \otimes \widetilde L^m \twoheadrightarrow
\widetilde E \otimes \widetilde L^m.
$$
The analog of projection formula as in global case leads to
$$
\xymatrix{
E \otimes L^m \otimes \sigma_{\ast}
(\sigma^{-1} I\cdot \OO_{\widetilde S})^m \ar@{^(->}[d] \\
\sigma_{\ast}(\sigma^{\ast}E \otimes \widetilde L^m)\ar[r]
& \sigma_{\ast}(\widetilde E \otimes \widetilde L^m)}
$$
and taking global sections
$$
\xymatrix{H^0(S,E \otimes L^m \otimes \sigma_{\ast}
(\sigma^{-1} I\cdot \OO_{\widetilde S})^m )\ar@{^(->}[d]
\ar[dr]^{\eta_t}\ar@{^(->}[r]&H^0(S,E \otimes L^m)\ar[d]_{\widetilde \eta_t} \\
H^0(S,\sigma_{\ast}(\sigma^{\ast}E \otimes \widetilde L^m))\ar[r]
& H^0(\widetilde S,\widetilde E \otimes \widetilde L^m)}
$$
Here $\widetilde \eta_t$ is included into the commutative diagram
$$
\xymatrix{H^0(S,E \otimes L^m)\ar[d]_{\widetilde \eta_t}\ar[r] &
H^0(t\times S\cap \Sigma_0),E \otimes L^m)\ar@{=}[d]\\
H^0(\widetilde S,\widetilde E \otimes \widetilde L^m) \ar[r]&
H^0(\widetilde S\cap \widetilde \Sigma_0,
\widetilde E \otimes \widetilde L^m)}
$$
where both horizontal arrows are restriction maps and upper
restriction map is injective. Hence $\widetilde \eta_t$ is also
injective. Since it is a monomorphism of vector spaces of equal
dimensions is is an isomorphism. Then $\eta: \V_m \to
\widetilde \V_m $ is also an isomorphism. The isomorphism
$\widetilde \V_m \cong \V'_m$ has been proven in previous
section (proof of Proposition \ref{hpol}).
\end{proof}

Identifying locally free sheaves $\V_m=\V'_m$ consider relative
Grothendieck' scheme $\Quot^{rp(n)}_T (p^{\ast} \V_m \otimes \L^{-m})$ and
two morphisms of closed immersion
$$
T\times S \stackrel{\widetilde{ev}}{\hookrightarrow}
\Quot^{rp(n)}_T (p^{\ast} \V_m \otimes \L^{-m}) \times S
\stackrel{\widetilde{ev'}}{\hookleftarrow} T\times S.
$$
Morphism $\widetilde{ev}$ is induced by the morphism
$ev: p^{\ast}\V_m \otimes \L^{-m}\twoheadrightarrow \E$ and
$\widetilde{ev'}$ by the morphism
$ev': p^{\ast}\V_m \otimes \L^{-m}\twoheadrightarrow \E'$.

Since both morphisms $\widetilde{ev}$ and $\widetilde{ev'}$ are proper
and they coincide along $\Sigma_0$ such that
$\codim_{T\times S} (T\times S) \setminus \Sigma_0\ge 2$,
then $\widetilde{ev}=\widetilde{ev'}$ and
$\widetilde{ev}(T\times S)=\widetilde{ev'}(T\times S)$
in scheme sense. Hence by universality of $\Quot$-scheme
$\E=\E'$ as inverse images of universal quotient sheaf over
$\Quot^{rp(n)}_T (p^{\ast} \V_m \otimes \L^{-m})$ under morphisms
$\widetilde{ev}=\widetilde{ev'}.$

Now turn to {\it 1 b)}. For $m\gg 0$ there are surjective morphisms
\begin{eqnarray*}
H^0(\widetilde S, \widetilde E\otimes \widetilde L^m) \otimes
\OO_{\widetilde S} \twoheadrightarrow \widetilde E \otimes
\widetilde L^m, \\
H^0(\widetilde S', \widetilde E'\otimes \widetilde L^{'m})
\otimes \OO_{\widetilde S'} \twoheadrightarrow \widetilde E'
\otimes \widetilde L^{'m}
\end{eqnarray*}
where $H^0(\widetilde S, \widetilde E\otimes \widetilde L^m)\cong
H^0(\widetilde S', \widetilde E'\otimes \widetilde L^{'m})=V_m$,
$\dim V_m=rp(m)$.
There are two induced closed immersions of schemes
$\widetilde S$ and $\widetilde S'$ into Grassmann variety
$$
\widetilde S \stackrel{j}{\hookrightarrow} G(rp(m), r) \stackrel{j'}{\hookleftarrow} \widetilde S'.
$$
Images of both schemes coincide off their additional components, i.e.
$$j(\widetilde S \setminus \bigcup_{i>0}\widetilde S_i)=
j'(\widetilde S' \setminus \bigcup_{i>0}\widetilde S'_i).$$ Hence
$j(\widetilde S_0)=j'(\widetilde S'_0)$. Since
$\sigma_0=\sigma|_{\widetilde S_0}$ as well as
$\sigma'_0=\sigma'|_{\widetilde S'_0}$ is a blowup morphism,
these blowup morphisms are defined by the same sheaf of ideals
$I\subset \OO_S.$ This leads to the conclusion that schemes $\widetilde S$
and $\widetilde S'$ in whole are defined by the same sheaf of ideals $I$
and hence $\widetilde S \cong \widetilde S'.$

Since $\widetilde L=L\otimes \sigma^{-1} I\cdot \OO_{\widetilde S}$
and $\widetilde L'=L\otimes \sigma^{'-1} I\cdot \OO_{\widetilde S'}$
where $\widetilde S \cong \widetilde S'$, $\sigma =\sigma'$, then
$\widetilde L \cong \widetilde L'.$

Now it rests to confirm that $\widetilde E \cong \widetilde E'.$
It will follow from the global consideration when $T=\Spec k$.

For global version consider families
$(\pi: \widetilde \Sigma \to T, \widetilde \L)$
and $(\pi': \widetilde \Sigma' \to T, \widetilde \L')$ and
epimorphisms
$\pi^{\ast}\pi_{\ast}\widetilde \L \twoheadrightarrow
\widetilde \L$, ${\pi'}^{\ast}\pi'_{\ast}\widetilde \L'
\twoheadrightarrow \widetilde \L'$.
We can assume that $\pi_{\ast} \widetilde \L =
\pi'_{\ast} \widetilde \L'$ and then identify projective bundles
$\P(\pi_{\ast} \widetilde \L)^{\vee}
=\P(\pi'_{\ast} \widetilde \L')^{\vee}$. Closed immersions of
$T$-schemes
\begin{eqnarray*}
j:\widetilde \Sigma \hookrightarrow \P(\pi_{\ast} \widetilde \L)^{\vee},\\
j':\widetilde \Sigma' \hookrightarrow \P(\pi_{\ast} \widetilde \L')^{\vee}
\end{eqnarray*}
and diagonal immersion $\P_{\Delta} \hookrightarrow
\P(\pi_{\ast} \widetilde \L)^{\vee} \times_T \P(\pi_{\ast}
\widetilde \L')^{\vee}$ lead to the commutative diagram
$$
\xymatrix{
& \P_{\Delta} \ar@{^(->}[d]\\
\P(\pi_{\ast} \widetilde \L)^{\vee}&\ar[l] \P(\pi_{\ast} \widetilde \L)^{\vee} \times_T \P(\pi_{\ast}
\widetilde \L')^{\vee} \ar[d] \ar[r]& \P(\pi_{\ast} \widetilde \L')^{\vee}\\
\widetilde \Sigma \ar@{^(->}[u]_j \ar[r]^\pi& T&\ar[l]_{\pi'} \widetilde \Sigma' \ar@{^(->}[u]^{j'}}
$$
Fibred product $\widetilde \Sigma \times _T \widetilde \Sigma' \hookrightarrow \P(\pi_{\ast} \widetilde \L)^{\vee} \times_T \P(\pi_{\ast}
\widetilde \L')^{\vee}$ gives rise to the intersection subscheme
$\widetilde \Sigma_{\Delta}= (\widetilde \Sigma \times _T
\widetilde \Sigma') \cap \P_{\Delta}.$ Now observe that there is a commutative
square
$$\xymatrix{\widetilde \Sigma \ar[d]_\pi&\ar[l] \widetilde \Sigma_\Delta \ar[d]\\
T &\ar[l]^{\pi'} \widetilde \Sigma'}
$$
where, as we have seen before, for any closed point $t\in T$ corresponding fibres of schemes
$\widetilde \Sigma$, $\widetilde \Sigma'$,
$\widetilde \Sigma_\Delta$ are identified isomorphically by arrows
of the diagram.

Also from the commutative diagram
$$\xymatrix{\widetilde \Sigma_\Delta \ar[d] \ar@{^(->}[r]^{j_\Delta} & \P_\Delta \ar[d]_\wr\\
\widetilde \Sigma \ar@{^(->}[r]^j& \P(\pi_{\ast} \widetilde \L)^{\vee}}
$$
we conclude that the left hand side vertical arrow is a closed immersion.
By the same reason there is a closed immersion $\widetilde \Sigma _\Delta
\hookrightarrow \widetilde \Sigma'.$
Now we make use of the following algebraic result.
\begin{proposition}\label{secflat}\cite[ch. 1, Proposition 2.5]{Milne} Let $B$
be flat $A$-algebra and $b\in B$. If the image of $b$ in $B/{\mathfrak m}B$
is not a zero divisor for any maximal ideal ${\mathfrak m}$ in $A$ then
$B/(b)$ is flat $A$-algebra.
\end{proposition}

Take a section $(s,s)$ of $\OO_{\P(\pi_{\ast}\widetilde \L)^{\vee}
\otimes_T \P(\pi'_{\ast}\widetilde \L')^{\vee}}$ and let
$b=(s', s'')$ be its image in $\OO_{\widetilde \Sigma} \otimes _T
\OO_{\widetilde \Sigma'}$. In our situation ${\mathfrak m}=
{\mathfrak m}_t$, and $b$ has an image in $\OO_{\pi^{-1}(t)}
\otimes_{k_t} \OO_{{\pi'}^{-1}(t)}$ which is not zero divisor.
Iterating the usage of the Proposition {\ref{secflat}} in regular
sequence \cite[ch. 1, Remark 2.6 d]{Milne}  one comes to
conclusion that $\widetilde \Sigma_\Delta$ is flat over $T$.
It rests to compare Hilbert polynomials of fibres over $t$ in the
exact triple
$$
0\to \II_{\widetilde \Sigma_\Delta, \widetilde \Sigma} \to
\OO_{j(\widetilde \Sigma)} \to \OO_{j_\Delta(\widetilde \Sigma_\Delta)} \to 0.
$$
Since $\OO_{\widetilde \Sigma}$ and $\OO_{\widetilde \Sigma_\Delta}$
are $T$-flat then $\II_{\widetilde \Sigma_\Delta, \widetilde \Sigma}$
is also $T$-flat. By  infinitesimal criterion of flatness
\cite{TimAA} fibrewise Hilbert polynomials
$\chi(\II_{\widetilde \Sigma_\Delta, \widetilde \Sigma} \otimes
\widetilde \L^n|_{j(\pi^{-1}(t))})$ do not depend on closed point $t\in T$.

We have $\chi(\OO(n)|_{j(\pi^{-1}(t))})=
\chi(\OO(n)|_{j_\Delta}(\pi_\Delta^{-1}(t)))$ and hence we conclude that
$\chi(\II_{\widetilde \Sigma_\Delta, \widetilde \Sigma} \otimes
\widetilde \L^n|_{j(\pi^{-1}(t))})=0$. Now $j$ and
$j_\Delta$ are identified under isomorphism
$\P(\pi_\ast \widetilde \L)^\vee = \P_\Delta.$ By the same reason
$j'$ and
$j_\Delta$ are identified under isomorphism
$\P(\pi'_\ast \widetilde \L')^\vee = \P_\Delta$ and hence
$\widetilde \Sigma \cong \widetilde \Sigma'$ and under this
identification also $\widetilde \L = j^{\ast}\OO(1) =
{j'}^\ast \OO(1)=\widetilde \L'$.

To confirm that also $\widetilde \E =\widetilde \E'$ we reason
in similar way and consider closed immersions of $T$-schemes
$$
\widetilde \Sigma \stackrel{j}{\hookrightarrow}
\Grass(\widetilde \V_m, r) = \Grass(\widetilde \V'_m, r) \stackrel{j'}{\hookleftarrow} \widetilde \Sigma'.
$$
Introduce shorthand notations $\G:=\Grass(\widetilde \V_m, r)$ and
$\G':=\Grass(\widetilde \V'_m, r)$ and form a fibred product
$\G \times _T  \G'$ together with the diagonal
$\G_\Delta \hookrightarrow \G \times _T  \G'$; and form the subscheme
$\widetilde \Sigma'_\Delta=(\widetilde \Sigma \times _T
\widetilde \Sigma') \cap \G_\Delta.$ As previously, there is a
commutative square
$$\xymatrix{\widetilde \Sigma'_\Delta \ar[d] \ar@{^(->}[r]^{j_\Delta} & \G_\Delta \ar[d]_\wr\\
\widetilde \Sigma \ar@{^(->}[r]^j& \G}
$$
from which
$\widetilde \Sigma'_\Delta \hookrightarrow \widetilde \Sigma$ as
closed subscheme. Applying Proposition \ref{secflat} we conclude
that $\widetilde \Sigma'_\Delta$ is flat over $T.$ Since fibres of
schemes $\widetilde \Sigma'_\Delta$ and $\widetilde \Sigma$  over
same closed point $t\in T$ coincide they have equal Hilbert
polynomials as subschemes in Grassmann variety $\G_t\cong G(rp(m), r).$
Hence $j_\Delta(\widetilde \Sigma'_\Delta )=j(\widetilde \Sigma)$
under identification $\G_\Delta =\G$ and
also $j_\Delta(\widetilde \Sigma'_\Delta )=j'(\widetilde \Sigma')$
under identification $\G_\Delta =\G'.$ Now let
$\pi^\ast \widetilde \V_m\twoheadrightarrow \QQ$ be the universal
quotient bundle on $\G=\G'$. Then
$\widetilde \E \otimes \widetilde \L^m=j^{\ast}\QQ =
{j'}^{\ast} \QQ =\widetilde \E' \otimes \widetilde L^m$ and hence
$\widetilde \E =\widetilde \E'.$

\end{document}